\documentclass{article}

\usepackage{
amsmath,
amsthm,
amscd,
amssymb,
wasysym
}
\usepackage{comment}
\usepackage{tikz}
\usetikzlibrary{positioning,arrows,calc}
\usepackage{graphicx}

\usepackage{url}
\theoremstyle{plain}
\newtheorem{lemma}{Lemma}
\newtheorem{corollary}{Corollary}

\newtheorem*{theorem*}{Theorem}
\newtheorem{theorem}{Theorem}

\newtheorem{question}{Question}
\newtheorem{remark}{Remark}

\newtheoremstyle{derp}% <name>
{3pt}% <Space above>
{3pt}% <Space below>
{}% <Body font>
{}% <Indent amount>
{\bfseries}% <Theorem head font>
{:}% <Punctuation after theorem head>
{.5em}% <Space after theorem headi>
{}% <Theorem head spec (can be left empty, meaning `normal')>

\theoremstyle{derp}
\newtheorem{example}{Example} %\textbf{Example}}

\theoremstyle{example}
\newtheorem{definition}{Definition}

\newcommand{\Z}{\mathbb{Z}}

\newcommand{\N}{\mathbb{N}}

\newcommand\xqed[1]{
  \leavevmode\unskip\penalty9999 \hbox{}\nobreak\hfill
  \quad\hbox{#1}}
\newcommand\qee{\xqed{$\fullmoon$}}

\newcommand{\End}{\mathrm{End}}
\newcommand{\Aut}{\mathrm{Aut}}

\newcommand{\Hom}{\mathrm{Hom}}

\newcommand{\SSS}{\mathcal{S}}
\newcommand{\id}{\mathrm{id}}

\begin{document}

%\markboth{Ville Salo}{Graph and wreath products of cellular automata}


\title{Graph and wreath products of cellular automata}

\author{Ville Salo \\ %\thanks{Research supported by the Academy of Finland grant 2608073211.} \\
University of Turku, Finland, \\
vosalo@utu.fi}
%\address{Department of Mathematics and Statistics, University of Turku, Finland}
%\makeatletter
%\email{vosalo@utu.fi}
%\makeatother

\maketitle

%\begin{history}
%\received{(Day Month Year)}
%\accepted{(Day Month Year)}
%\comby{[editor]}
%\end{history}

\begin{abstract}
We prove that the set of subgroups of the automorphism group of a two-sided full shift is closed under countable graph products. We introduce the notion of a group action without $A$-cancellation (for an abelian group $A$), and show that when $A$ is a finite abelian group and $G$ is a group of cellular automata whose action does not have $A$-cancellation, the wreath product $A \wr G$ embeds in the automorphism group of a full shift. We show that all free abelian groups and free groups admit such cellular automata actions. In the one-sided case, we prove variants of these results with reasonable alphabet blow-ups.
\end{abstract}
%\keywords{Cellular automata; Automorphism group; Graph product; Wreath product}

%\ccode{Mathematics Subject Classification 2010: 37B10 (primary), 37B15, 20B27 (secondary)}

% F_2 x F_2 has subgroup H with undecidable word problem
% therefore its distortion function has to be arecursive?
% suppose some computable function t upper bounds it...
% given w, we have some length in F_2 x F_2
% 

% length n but takes n^2 to draw... 
%

% how to show that closed under graph products?
% do the same construction but if G_i and G_j are supposed to commute, then
% we do not allow the j-controlled permutations on right sides of i-part and vice versa.
% then indeed G_i and G_j commute, because 
% if you first act by G_i it changes the i things and possibly guys left of them, then
% G_j does same... and back...
% ok best way is to analyze what commutator does in each separate...
% in G_i-slots, G_i forward, G_i back.
% in G_j slots, same
% in others, rightmost symbol goes forward and back, at most.

% suppose now that something is nontrivial...
% we prove that for any guy b who can be permuted to the right, there is a word with 
% b-area as rightmost, such that

% if some guy commutes with all, then element acts nontrivially and moves things at max speed: have only a word of that type.

% write the element as wabu where b doesn't commute with a, but commutes with
% all of u.
% wa is nontrivial so acts with maximal rate one some word with a as leftmost

\section{Introduction}

A recent trend in symbolic dynamics is the study of automorphism groups of subshifts, both in ``simple subshifts'' of various kinds and in situations closer to a full shift \cite{Ol13,SaTo15d,Sa14d,CyKr16a,CoQuYa16,CyKr16b,DoDuMaPe16,DoDuMaPe17,CyFrKrPe18,FrScTa19,Sa18d,Sa19a,BaKaSa16,HaKrSc21}. One interesting question in all of the settings is which (kinds of) groups can be abstractly embedded in automorphism groups of subshifts, as a function of dynamical restrictions put on the subshift.

The most classical, and in some sense simplest, case is that of (reversible) cellular automata: what groups can be embedded as subgroups of automorphism groups of full shifts $A^\Z$? The class of groups that embed is known to be quite rich, but it is not very well understood. The study of this family $\mathcal{G}$ began in Hedlund's 1969 paper \cite{He69} where it was shown that finite groups and the infinite dihedral group are in $\mathcal{G}$. It is known \cite{KiRo90} that automorphism groups of full shifts embed in each other, so $\mathcal{G}$ does not depend on the alphabet of the full shift.

For this paper, the most relevant facts known about $\mathcal{G}$ are the following:
\begin{enumerate}
\item $\mathcal{G}$ is closed under countable direct sums and free products, \cite{Sa18d}
\item $\mathcal{G}$ contains all graph groups, \cite{KiRo90}
\item $\mathcal{G}$ contains the lamplighter group $\Z_2 \wr \Z$, \cite{Sa18a}
\end{enumerate}

For context, we list some other known facts about $\mathcal{G}$. It is commensurability invariant \cite[Proposition~3.1]{KiRo90}, groups in $\mathcal{G}$ are residually finite \cite{BoLiRu88}, finitely generated (f.g.) groups in $\mathcal{G}$ have word problem in co-NP and this problem is co-NP-complete for some f.g.\ $G \in \mathcal{G}$ \cite{Sa20}, the locally finite groups in $\mathcal{G}$ are exactly the residually finite countable ones \cite{KiRo90}, $\mathcal{G}$ has a f.g.\ subgroup with no free subgroups which is not virtually solvable (i.e.\ the Tits alternative fails) \cite{Sa19a,Sa18a}, f.g.\ groups in $\mathcal{G}$ may have undecidable torsion problem \cite{Sa18a}, and there is a f.g.-universal f.g.\ group in $\mathcal{G}$ (a f.g.\ group which contains a copy of every f.g.\ group in $\mathcal{G}$) \cite{Sa18a}.

From these facts, many additional results follow, for instance the closure properties imply that $\mathcal{G}$ contains all finite groups (first proved in \cite{He69}), thus the free product of all finite groups (first proved in \cite{Al88}), thus $\Z_2*\Z_2$, thus $\Z$, thus all countable free groups (first proved in \cite{BoLiRu88}) and finitely-generated abelian groups. Graph groups and commensurability invariance imply that fundamental groups of $2$-manifolds are in $\mathcal{G}$ \cite{KiRo90}. It also follows that $\mathcal{G}$ is not quasi-isometry invariant since $F_2 \times F_2$ is a graph group but is quasi-isometric to a group that is not residually finite \cite{BuMo00}.

To our knowledge, residual finiteness and the complexity restriction of the word problem are the only known restrictions for f.g.\ groups to be in $\mathcal{G}$.

Our aim in this paper is to clarify the situation in the ``tame end''\footnote{``Dynamically'', already f.g.\ subgroups of $F_2 \times F_2$ can be very wild. For example, they can have an undecidable conjugacy problem \cite{Mi58,LySc15} and can have arbitrarily badly distorted subgroups \cite{OlYuSa01,SaBiRi02}. However, both finite graph products and wreath products preserve polynomial-time decidability of the word problem (for the former, this is clear from the normal form \cite{Gr90}), which is atypical for groups of cellular automata if P $\neq$ NP. Also, the Tits alternative is preserved by a large class of graph products \cite{AnMi15}.} of the complexity spectrum of $\mathcal{G}$, by clarifying the three items listed above. First, we combine the first two items, graph groups and closure under direct and free products, into a single closure property: 

\begin{theorem}
The class $\mathcal{G}$ is closed under countable graph products.
\end{theorem}

The finite case of the construction is similar to \cite{Sa18d}, mixing what was done in the direct and free product cases. Kim and Roush proved in \cite{KiRo90} the cases of finite graph products where the node groups are finite or infinite cyclic. In the general case of countable products, some technical difficulties arise which do not have analogs in the previous papers, since the individual generators can no longer understand the global picture. A reference on graph products is \cite{Gr90}.

We also generalize the third item, the lamplighter group.

\begin{theorem}
If $A$ is a finite abelian group and $G \in \mathcal{G}$ acts on a full shift by automorphisms without $A$-cancellation, then $A \wr G \in \mathcal{G}$.
\end{theorem}

Here, \emph{acting without $A$-cancellation} is a technical notion, which roughly means that the occurrences of some patterns, with some chosen ``weights'' in $A$, do not exhibit cancellation in $A$ along any sequence of translations. A reference on wreath products is \cite{BhMoMaNe06}. %means roughly that the trace subshift (what you can see at the origin along $G$-orbits) has no cancellation w.r.t.\ the group $A$, although we allow a bit more freedom than this (see Definition~\ref{def:Aithful}).

We complement this result with many examples of groups that act without $A$-cancellation by automorphisms, obtaining for example that the groups $A \wr \Z^d$ and $A \wr F_d$ are in $\mathcal{G}$ for any finite abelian group $A$ and $d \in \N$.

While lack of $A$-cancellation is precisely the notion needed (in our construction) for wreath products, in all our concrete examples we deduce it from the stronger property of \emph{unique visits} (see Definition~\ref{def:StronglyFaithful}), which may be of independent interest. %, which we believe is of independent interest, and has links to other dynamical properties (see Lemma~\ref{lem:SFImplications}). %This property is implied by having a ``sunny-side-up subquotient'', and if $G$ is finitely-generated it implies that ``some point visits some clopen set just once in a horoball''. We explain these notions in . %. (see Definition~\ref{def:StronglyFaithful}).

Similar constructions work in the more restricted setting of one-sided automorphism groups, though (at least without modification) we obtain somewhat weaker statements. Let $\mathcal{G}_n'$ be the class of subgroups of $\Aut(\Sigma^\N)$ where $|\Sigma| = n$, and let $\mathcal{G}_\infty' = \bigcup_n \mathcal{G}_n'$.

It is easy to see that $\mathcal{G}_\infty' \subset \mathcal{G}$ (as we explain in the following section, $\Aut(A^\N)$ can be seen as a subgroup of $\Aut(A^\Z)$). It is known that $\mathcal{G}_n'$ \emph{does} depend on $n$; in fact every finite group is in $\mathcal{G}_\infty'$, but no $\mathcal{G}_n'$ contains every finite group. See \cite{BoFrKi90} for basic information on $\Aut(A^\N)$ and its subgroups.

We obtain the following results:

\begin{theorem}
Any finite graph product of groups in $\mathcal{G}_n'$ is in $\mathcal{G}_{n+1}'$.
\end{theorem}

\begin{theorem}
If $A$ is finite abelian and $G \in \mathcal{G}'_\infty$ has non-$A$-cancellation, then $A \wr G \in \mathcal{G}'_\infty$.
\end{theorem}

\begin{corollary}
If $A$ is finite abelian and $n \in \N$, then $A \wr F_n, A \wr \Z^n \in \mathcal{G}'_\infty$.
\end{corollary}

\section{Definitions and conventions}

We take $0 \in \N$, $\Z_+ = \N \setminus \{0\}$. Every finite set and every group has the discrete topology, and the function space $A^B$ has the compact-open topology, which is always also the product topology. Throughout, $\Sigma$ is a finite set with at least two elements, called the \emph{alphabet}; and $\Sigma^*$ is the set of words over this alphabet (which can be taken to be the free monoid generated by the alphabet). The elements of $\Sigma$ are called \emph{letters} or \emph{symbols}. We use the notation $A \Subset B$ for finite subsets ($A \subset B \wedge |A| < \infty$). By $G$ we denote a group, and $e = e_G \in G$ is always an identity element. Groups act from the left.

Words are $0$-indexed and for a word $u \in \Sigma^*$ write $u^\bot$ for its reverse $u^\bot_i = u_{|u|-1-i}$, and $|u|$ for its \emph{length} (the number of symbols it is composed of, e.g.\ $|aa| = 2$). Write the \emph{empty word} (the unique word of length $0$) as $\varepsilon$, write concatenation of words $u, v$ as $u \cdot v$ or simply $uv$, define $u^0 = \varepsilon$ and $u^{n+1} = uu^n$. For sets of words $A, B$, write $AB = \{uv \;|\; u \in A, v \in B\}$, $A^0 = \{\varepsilon\}$ and $A^{n+1} = AA^n$. Write $w^*$ for $\{w^n \;|\; n \in \N\}$ and $A^* = \bigcup_{n \in \N} A^n$. For a letter $a \in \Sigma$, we sometimes write $a \in w$ for $\exists j: w_j = a$.

A \emph{configuration} is a bi-infinite word $x \in \Sigma^\Z$ (generalized below). Write $u^\Z$ for the configuration $x$ with $x_i = u_j$ where $j \equiv i \bmod |u|$. For $u,v \in \Sigma^*$ write
\[ u \sqsubset v \iff \exists j \in \{0, \dots, |v|-|u|\}: \forall i \in \{0,\dots,|u|-1\}: v_{j+i} = u_i. \]

A set of words $W \subset \Sigma^*$ is \emph{mutually unbordered} if $\forall u, v \in W: uw = w'v \implies w = w' = \varepsilon \vee |w| \geq |v|$. The point of this definition is that whenever words from a mutually unbordered set of words appear in a configuration $x \in \Sigma^\Z$, there are no overlaps between their occurrences. The set $\Sigma$ itself is mutually unbordered, and for example the set of all words that begin with a particular letter and do not otherwise contain it (e.g.\ $2\{0,1\}^*$) is also mutually unbordered.

A \emph{(zero-dimensional discrete topological dynamical) $G$-system} is a pair $(G, X)$ where $G$ is a discrete group, $X$ is a zero-dimensional compact metrizable space, and $G$ acts on $X$ by continuous maps. Note that the action is of course part of the data, but is not included in the notation. On the other hand, specifying $G$ is superfluous when talking about $G$-systems, so we sometimes simply say $X$ is a $G$-system. (These somewhat strange conventions are standard.)

We denote the action of $g \in G$ on $x \in X$ by just $gx$. The elements $x \in X$ are called \emph{points}. A \emph{factor} of a system is another system $(G,Y)$ such that there is a continuous surjection $f : X \to Y$ which intertwines the actions as $f(gx) = gf(x)$ for all $x \in X, g \in G$, and an \emph{isomorphism} of dynamical systems is a factor map that has a two-sided factor map inverse. A \emph{subsystem} is a closed $G$-invariant subset.

A \emph{subshift} is a subsystem of the \emph{full shift} $A^G$ where $G$ acts by $gx_h = x_{hg}$. Points of subshifts are also called \emph{configurations}; note that in the case $G = \Z$ they correspond to configurations as defined above. Subshifts are, up to isomorphism, the systems where the action is \emph{expansive}, meaning there exists $\varepsilon > 0$ such that $x \neq y \implies \exists g \in G: d(gx, gy) > \varepsilon$. We use these meanings of ``subshift'' rather interchangeably. A system is \emph{faithful} if $(\forall x \in X: gx = x) \implies g = e$.

A subshift $X \subset \Sigma^G$ is an \emph{SFT} if there exists a clopen set $C \subset \Sigma^G$ such that $X = \bigcap_{g \in G} g C$. Clopen sets are sets that are defined by looking at only finitely many coordinates of $G$, and thus SFTs are sets defined by a finite local constraint which is checked uniformly on the group $G$ in every pattern of shape $Dg$ as $g$ ranges over $G$, for a finite subset $D \Subset G$. \emph{Sofic} shifts are subshifts which are factors of SFTs. Of course full shifts are sofic.

The \emph{(one-dimensional) full shift} is the $\Z$-full shift $\Sigma^\Z$, and for clarity we write the action of its generator $1 \in \Z$ (which determines the entire action) as $\sigma(x)_i = x_{i+1}$. Its shift-commuting continuous self-maps are known as \emph{endomorphisms} or \emph{cellular automata}, and the ones that are injective (equivalently, have a left and right inverse) are called \emph{automorphisms} or \emph{reversible cellular automata}. We only consider reversible cellular automata in this paper. Reversible cellular automata form a group denoted $\Aut(\Sigma^\Z)$.

\begin{definition}
By $\mathcal{G}$ we denote the set of isomorphism classes of subgroups of $\Aut(\Sigma^\Z)$.
\end{definition}

This definition does not depend on $\Sigma$ \cite{KiRo90} (as long as $|\Sigma| \geq 2$). We also call groups in $\mathcal{G}$ \emph{groups of cellular automata}. Sometimes we talk about the \emph{natural action} of a group $G \in \mathcal{G}$, this means we assume $G$ is represented in some way as a concrete group consisting of cellular automata on some full shift $A^\Z$ and the natural action is its defining action on $A^\Z$.

By the Curtis-Hedlund-Lyndon theorem \cite{He69}, a cellular automaton $f : \Sigma^\Z \to \Sigma^\Z$ has a \emph{local rule} $F : \Sigma^{2r+1} \to \Sigma$ such that $f(x)_i = F(x|_{\{i-r,\ldots,i+r\}})$ for all $x \in \Sigma^\Z$. The \emph{minimal radius} is the minimal possible $r$, and for this $r$ there obviously exist $u \in \Sigma^{2r}$ and $a, b \in \Sigma$ such that either $F(au) \neq F(bu)$ or $F(ua) \neq F(ub)$. We often identify $f$ with $F$, and more generally apply it to words of length $2r+k$ to produce words of length $k$, by applying it to the consecutive length $2r+1$ subwords.

A cellular automaton is \emph{one-sided} if we can pick $F : \Sigma^{r+1} \to \Sigma$ such that $f(x)_i = F(x|_{\{i,i+1,\ldots,i+r\}})$ for all $x \in \Sigma^\Z$ and $i \in \Z$. The shift-commuting endomorphisms of $\Sigma^\N$ under the $\N$-action by $\sigma(x)_i = x_{i+1}$ are in an obvious one-to-one correspondence with one-sided cellular automata on $\Sigma^\Z$, and the ones that are bijective (as self-maps of $\Sigma^\N$) again form a group $\Aut(\Sigma^\N)$. The group $\Aut(\Sigma^\N)$ can be identified with the subgroup of $\Aut(\Sigma^\Z)$ consisting of reversible cellular automata $f$ such that both $f$ and $f^{-1}$ are one-sided.

Usually, we define a cellular automaton by explaining what its action is on a dense set of configurations. Dense refers to the topological meaning, i.e.\ a dense set $Y \subset \Sigma^\Z$ is just a set such that for every $n\in\N$, every word of length $2n+1$ appears as $y|_{\{-n, \ldots, n\}}$ for some $y \in Y$.

Describing the rule on a dense set is easier than giving a local rule, because in the local rule we only see what happens in individual cells, while the behavior of nearby cells is usually highly coordinated and is really just a coding (application through a suitable conjugating map) of some natural action. We want to make sure that the cells agree on the interpretations, and the best way to do this is to only describe the rule directly in terms of such an interpretation. The ``dense'' assumption is mainly a matter of presentation: it allows us to assume that no behavior (like a repetition of symbols from some subset of the alphabet) continues infinitely, avoiding the need to discuss infinite tails separately.

On the other hand, when the rule is given by describing the action, it is not automatic that it is shift-commuting (but it is if we only discuss coordinates relatively). It is also not automatic that it extends to a continuous map, and thus we need to be sure that our maps defined on the dense set are in fact \emph{uniformly continuous}, i.e.\ we can give some bound on how far we need to look to determine the image at some cell. However, this should in practice be completely obvious in each case. %This is discussed in more detail in \cite{Sa18d}.
On the third hand, when the rule is given by describing a natural bijective action directly, reversibity is trivial.

%The identity element of an abstract group $G$ is $e = e_G$.

Conjugation in a group is $g^h = h^{-1}gh$. The \emph{free group} on $n$ free generators is $F_n = \langle g_1,g_2,\ldots,g_n \rangle$. The cyclic group with $n$ elements is $\Z_n$, usually written additively. Group elements are sometimes called \emph{cells}, especially when $i \in \Z$ and working with $\Z$-subshifts. The groups $G$ and $H$ are \emph{commensurable} if they share a finite index subgroup (up to isomorphism).

The \emph{(restricted) wreath product} of groups $A$ and $G$, which we may assume disjoint apart from the identity, is the group $A \wr G$ with the presentation
\[ \langle A, G \;|\; \forall g, h \in G: \forall a, b \in A: (g \neq h \implies [a^g, b^h] = e) \rangle \]
where it is understood that the relations of $G$ and $A$ also hold, and we recall that $e$ is the identity element.

Graph products are defined as follows: Let $\Gamma = (V, E)$ be a graph, i.e.\ $E \subset \{\{u, v\} \;|\; u, v \in V, u \neq v \}$. Losing no generality, we assume $V = \{1,2,\ldots,n\}$ or $V = \N$. Let $\overline{G} = (G_1, \ldots, G_n)$ or $\overline{G} = (G_i)_{i \in V}$ be groups, which we assume disjoint apart from sharing the identity element. Then we write $\overline{G}^\Gamma$ for the corresponding \emph{graph product}
\[ \overline{G}^\Gamma = \langle \bigcup_{i \in V} G_i \;|\; \forall \{i, j\} \in E: [G_i, G_j] \rangle, \]
where $[G_i, G_j] = \{[a,b] \;|\; a \in G_i, b \in G_j\}$, where it is understood that the relations of the groups $G_i$ also hold. We call the groups $G_i$ the \emph{node groups}.

Readers unfamiliar with graph products may wonder whether there can be ``unintended consequences'' of such relations beyond what it says on the tin (i.e.\ make the groups connected by edges commute). In a sense there are none, in that if a word over elements of different groups cannot be shortened in the obvious way by permuting commuting group elements so that two elements from the same group $G_i$ can be joined, then it must be nontrivial. This follows from the normal form given in \cite{Gr90}. We recover this result for groups in $\mathcal{G}$ in the proof of Theorem~\ref{thm:GraphProducts}.

\subsection{Conveyor belt construction}

We recall the basic conveyor belt construction, whose variants will be used throughout the paper (and which was first used under the term bucket-passing in \cite{KiRo90}). Consider the group $\Aut(\Sigma^\Z)$ for a finite alphabet $\Sigma$. Then since automorphisms commute with the shift map, $\Aut(\Sigma^\Z)$ acts naturally on the fixed points of $\sigma^n$, i.e.\ the set of $n$-periodic points.

In particular, there is a natural action on the set $W_n = (\Sigma^n)^2$ of pairs of words of length $n$, namely
\[ f(u, v) = (f((uv)^\Z)|_{\{0, \ldots, n-1\}}, f((uv)^\Z)|_{\{n, \ldots, 2n-1\}}). \]

Finally, we can write $(u, v) \in W_n$ as a single word $w \in (\Sigma^2)^n$ by defining $w_i = (u_i, v^{\bot}_i)$. Note that this is not the most obvious way of identifying $(\Sigma^n)^2$ and $(\Sigma^2)^n$ as $v$ is reversed. To specify that the identification is being made this way we call $w$ a \emph{conveyor belt} (corresponding to the periodic point $(uv)^\Z$).

This terminology comes from a particular geometric picture. Indeed, the left shift action $\sigma$ on the orbit of $(uv)^\Z$ conjugates in this map to a counterclockwise ``rotation'' of $w_i$, where the \emph{top track} containing $u$ moves left, the \emph{bottom track} containing $v^\bot$ moves right, and at the boundaries symbols move from bottom up or top down. Thus $w$ acts as a conveyor belt in this conjugated action.

Now it easy to verify that the natural action of $\Aut(\Sigma^\Z)$ on $W_n$ conjugates, on the conveyor belts, to an action given by local rules (like those of cellular automata, but aware of the boundaries). We call this the \emph{conveyor belt action}. To determine the symbol $k$ steps to the left from a position on the top track, we move to the left along the top track, until we reach the left end, move to the bottom track and start traveling right. With this logic (adapted in an obvious way for the bottom track, and for accessing right neighbors) we can determine the neighborhood as it would be seen through the conjugacy with periodic points.

This gives, for any finite set $\Gamma$, a natural embedding of $\Aut(\Sigma^\Z)$ into $\Aut((\Sigma^2 \cup \Gamma)^\Z)$. Namely, on the dense set of configurations where all maximal sequences of symbols from $\Sigma^2$ are finite, we apply the conveyor belt action. This gives a shift-commuting rule (since we just shift the finite conveyor belts and apply the same actions!) which is uniformly continuous (because the local rules described in the previous paragraph have the same radius as the original maps when we know that all conveyor belts are finite). Therefore, there is a unique extension to all of $(\Sigma^2 \cup \Gamma)^\Z$, with all of $\Gamma^\Z$ consisting of fixed points.

We will sometimes call symbols from $\Sigma^2$ \emph{double letters} to avoid confusion with symbols from $\Sigma$, and similarly a word over the alphabet $\Sigma^2$ is sometimes called a \emph{double word}. %, as often it is useful to pass back and forth between actions on conveyor belts and periodic points.

\section{Closure under graph products}

\begin{theorem}
\label{thm:GraphProducts}
The class $\mathcal{G}$ is closed under countable graph products.
\end{theorem}

\begin{proof}
Consider first the case of a finite graph product $(G_1, G_2, \dots, G_n)^\Gamma$. We may assume each group $G_i$ acts faithfully by cellular automata on $\{0,1\}^\Z$, and that the groups $G_i$ are disjoint apart from sharing the identity. We may also assume that each $g \in G_i$ has minimal radius at least $1$, i.e.\ the symbol flip $0 \leftrightarrow 1$ is not in any of these groups. This is possible by using the non-trivial self-embeddings of $\mathcal{G}$ constructed in \cite{KiRo90}. %\footnote{Alternatively, one can use self-embeddings of $\Aut(\{0,1\}^\Z)$ to literally obtain disjoint groups $G_i \leq \Aut(\{0,1\}^\Z)$.}

Let $B = \{0,1\}^2$, $\SSS = B^B$, i.e.\ the set of functions from $B$ to $B$, and pick the alphabet $\Sigma = (B \times \{1,\dots,n\}) \cup \SSS$. On $B$ we pick some abelian group structure, for example through the natural identification $B \cong \Z_2^2$.

We can see a word over the alphabet $B \times \{1,\dots,n\}$ as a triple $(u,v,w)$ where $u, v \in \{0,1\}^*$ and $w \in \{1,\dots,n\}^*$ and $|u| = |v| = |w|$ in an obvious way. If $g \in G_i$, then define a map $\hat g \in \Aut(\Sigma^\Z)$ as follows: On a dense set, a configuration in $\Sigma^\Z$ splits into maximal finite subwords of the form $(u, v, i^m)$ (called \emph{segments}), with $u, v \in \{0,1\}^m$, $i \in \{1,\dots,n\}$, and into symbols in $\SSS$.

We define the action on these words in such a way that this extends to an automorphism, and for this we use the conveyor belt construction: words of the form $(u, v, i^m)$ are mapped to $(y_{[0,m-1]}, y_{[m,2m-1]}, i^m)$ where $y = g((uv^\bot)^\Z)$. Symbols in $\SSS$ are never modified, and the $\{1,\ldots,n\}$-components of states in $B \times \{1,\ldots,n\}$ are also never modified.

If $i \neq j$, then $\hat g$ acts on maximal words $(ua, vb, j^m)$ as follows, for $a, b \in \{0,1\}$: The words $u$ and $v$ are not modified. If there is an edge between $i$ and $j$ in $\Gamma$, $\hat g$ acts as identity. Suppose then there is no edge (so that the groups $G_i, G_j$ should \emph{not} commute). If the symbols immediately to the right of the segment of the tape containing $(ua, vb, j^m)$ are $sc$, then nothing is done unless $s \in \SSS$ and $c \in B \times \{i\}$. Finally, if $s \in \SSS$ and $c = (d,i) \in B \times \{i\}$, and if the symbol $(d,i)$ is changed to $(d',i)$ when $\hat g$ is applied in the segment on the right, then we change $(a, b, j)$ to $((a, b) + s(d') - s(d), j)$ (using the group structure of $B$). This is called a \emph{side-effect} of the application.

It is seen as in the previous section %(as in \cite{Sa18d,KiRo90}
that this gives a well-defined automorphism. Briefly, considering one of the individual groups $G_i$, the action on the segments over the subalphabet $B \times \{i\}$ mimics the periodic point action (meaning the action of $\Aut(\Sigma^\Z)$ on the set of points with a finite shift-orbit, in this case of even length), and the side-effects on the rightmost positions of words of the form $B \times \{j\}$ telescope to zero when the leftmost symbol of the neighboring $B \times \{i\}$ segment returns to its original value. Thus $\hat g^{-1} = \widehat{g^{-1}}$, and $g \mapsto \hat g$ gives an embedding of $G_i$. In particular we obtain an action of the free product of these groups.

We can extend the action of $\hat g$ for $g \in \bigcup_i G_i$ in a natural way to finite words over $\Sigma$, by the same description above (so that this agrees with the action on configurations for example if we think of the finite word as being surrounded by tails entirely over $\SSS$). The maximal words of the form $(u, v, i^m)$ which such a word splits into are again called its \emph{segments}, and $i$ is the \emph{type} of the segment. As is standard, we write elements of $(G_1,\dots,G_n)^\Gamma$ as words over the alphabet $\bigcup_i G_i \setminus \{e\}$, and such a word is \emph{reduced} if it cannot be made shorter by permuting commuting elements, joining subwords of the form $G_i G_i$ and removing occurrences of $e$. For a particular reduced word $w$ we call the individual symbols in $G_i$ its \emph{syllables}, and we can associate to a reduced word $w \in (\bigcup_i G_i \setminus \{e\})^*$ a word $\tau(w) \in \{1,2,\dots,n\}^*$ by only recording the types of syllables, where the type of $g \in G_i$ is $i$.

It is easy to check that if $g \in G_i, h \in G_j$, and $\{i,j\} \in E(\Gamma)$, then the cellular automata $\hat g$ and $\hat h$ defined above commute, as their actions never modify or depend on the same part of the tape. Thus $g \mapsto \hat g$ extends in a well-defined way to the graph product by the definition of the graph product (this can also be stated as the universal property of the graph product). We now consider an arbitrary reduced word $w$ of length at least one, and show that the corresponding automorphism is nontrivial, which proves that the action is faithful.

More precisely, we will prove the following. Let $w \in (\bigcup_i G_i \setminus \{e\})^*$ be any nonempty reduced word representing an element $g \in (G_1,\dots,G_n)^\Gamma$. Suppose either $\tau(w) = jv$, $j \in \{1,\dots,n\}$ where $\{\{j, k\} \;|\; k \in v\} \subset E(\Gamma)$, or we have $\tau(w) = jviu$ where $\{j, i\} \notin E(\Gamma)$ and $\{\{j, k\} \;|\; k \in v\} \subset E(\Gamma)$ (these cases of course cover all nonempty reduced words -- we are simply giving names to syllables and sequences of syllables). Then there is a word $t \in \Sigma^*$ whose leftmost segment is of type~$j$, such that some perturbation of the rightmost symbol of $t$ effects some change in the leftmost symbol of $\hat g(t)$. Furthermore, this change only happens at the ``last step'', when applying the syllables of $w$ one by one from right to left (i.e.\ the change only becomes visible when applying the group element of type $j$ in the beginning of $w$).

In formulas, we want that there exists another word $t' \in \Sigma^*$ with $|t'| = |t|$ whose leftmost segment is also of type~$j$, such that $t_{[0,|t|-2]} = t'_{[0,|t'|-2]}$, $\hat g(t)_0 \neq \hat g(t')_0$, and $\hat h(t)_0 = \hat h(t')_0$ if $h$ is a group element corresponding to a proper suffix of $w$.

The base case $\tau(w) = jv$ is obvious by the definition of the minimal radius $r$ of a cellular automaton (where recall that we assumed $r \geq 1$ for all the cellular automata in all the groups), by taking a word with just one suitably chosen segment of type~$j$ whose length is $r+1$. Namely, $j$ cannot appear in $v$ because every symbol in $v$ commutes with $j$ and $w$ is reduced.

The case $jviu$ is proved by induction. Write $w = w'w''$ with $\tau(w') = jv, \tau(w'') = iu$, with the assumptions above, and let $g', g''$ be the group elements corresponding to $w', w''$. (When considering the corresponding automorphisms, we'll write $\hat g'$ instead of $\widehat{g'}$ in what follows, for typographical reasons.) The subword $jv$ corresponds to the base case, and we can find words $t, t'$ of the same length which only differ in the last symbol, and which consist of a single type-$j$ segment, such that the leftmost symbols of $\hat g'(t)$ and $\hat g'(t')$ differ, and furthermore the first difference is seen at the last step. Suppose $t_{|t|-1} = (a,j)$ and $t'_{|t'|-1} = (b,j)$, where $a \in B^2, b \in B^2$.

Next, since $w'' = iu$ is reduced, by induction we can find words $t'', t'''$ which only differ in the rightmost symbol, whose leftmost segment is of type~$i$, such that $\hat g''(t'')_0 = (d,i) \neq (d',i) = \hat g''(t''')_0$, for some distinct $d \in B^2, d' \in B^2$ and the change only happens at the last step (in the sense explained above). Now recall that if we put an element $s \in \SSS$ between the word $t$ and a word $\hat t$ that begins with a segment of type $i$, then (because $\{j, i\} \notin E$), if a cellular automaton corresponding to an element of $G_i$ changes the initial symbol $d$ of $\hat t$ to $d'$, then we will as a side effect change the last symbol $(a,j)$ of $t$ to $(a + s(d') - s(d), j)$. Picking any $s \in \SSS$ such that $s(d') - s(d) = b - a$, we have $(a + s(d') - s(d), j) = (b, j)$, so that such a change turns $t$ into $t'$.

Let $T = (B^2 \times \{j\})^{|t|}$. Let $\hat t \in T$, and observe that
\[ \hat g''(\hat t \cdot s \cdot t'') = \phi(\hat t) \cdot s \cdot \hat g''(t'') \]
for some function $\phi : T \to T$. The length-$|t''|$ suffix is indeed independent of $\hat t$ since information never flows to the right over a symbol from $\SSS$, and from this we conclude that $\phi$ must be a bijection because $\hat g''$ is a bijection on words of a fixed length. Since $t \in T$ and $\phi$ is a bijection, there exists $\hat t \in T$ such that $\phi(\hat t) = t$. For this choice we have
\[ \hat g''(\hat t \cdot s \cdot t'') = t \cdot s \cdot \hat g''(t''). \]

We now claim that
\[ \hat g''(\hat t \cdot s \cdot t''') = t' \cdot s \cdot \hat g''(t'''). \]
For this, observe that after an application of $h$, a group element corresponding to a suffix of $w''$, the prefix $\hat t$ must evolve the same way, i.e.\ $\hat h(\hat t \cdot s \cdot t'')|_{\{0,\ldots,|\hat t|-1\}} = \hat h(\hat t \cdot s \cdot t'')|_{\{0,\ldots,|\hat t|-1\}}$. This is because the first symbols of the length-$|t''|$ suffixes are the same by the inductive assumption.

The last automorphism $h$ applied when applying $\hat g''$ (which is the leftmost syllable of $w''$), is then the first time this leftmost symbol changes. This $h$ is of type $i$, since $\tau(w'') = iu$. Since $\hat t$ is a segment of type $j$, the application of $h$ can only change the last symbol of this segment, and since
$\hat g''(t'')_0 = (d, i)$ and $\hat g''(t''')_0 = (d', i)$, the change in the last symbol changes $t$ to $t'$ by our choice of $s$. This concludes the proof that $\hat g''(\hat t \cdot s \cdot t''')|_{\{0, \ldots, |t'|-1\}} = t'$.

Then consider the applications 
\[ \hat g' \hat g''(\hat t \cdot s \cdot t'') = \hat g'(t \cdot s \cdot \hat g''(t'')) \]
and
\[ \hat g' \hat g''(\hat t \cdot s \cdot t''') = \hat g'(t' \cdot s \cdot \hat g''(t''')). \]
Here, $\hat g'$ corresponds to syllables of types $jv$. Since the syllables corresponding to $v$ commute with $i$, the $|t|$-prefixes $t$ and $t'$ of are not modified before the last step (i.e.\ they are not changed by the application of $\hat h$ where $h$ corresponds to a suffix of $\hat g'$). On the last step, the initial segments evolve to $\hat g'(t)$ and $\hat g'(t')$ respectively, and the initial symbols are different by the assumption on these words. This concludes the proof of the induction step, thus the proof of the technical claim, and thus the proof of the case of a graph product coming from a finite graph $\Gamma$.

\vspace{0.2cm}

\textbf{Infinite locally cofinite $\Gamma$.} Consider now the case of an infinite locally cofinite graph $\Gamma$, i.e.\ every node is neighbors with all but finitely many nodes, i.e.\ almost any two groups commute. Pick any nontrivial alphabet $\Sigma$ and for each $i$ pick a set of words $W_i \subset \Sigma^i$ with $|W_i| = 4$, such that $\bigcup_i W_i$ is mutually unbordered. Replace the use of $B \times \{i\}$ with $W_i$, i.e.\ a segment of type~$i$ is redefined to be a maximal finite word from $W_i^*$, which will again be interpreted as a conveyor belt built from two binary words. We can use the same $\SSS = B^B$ (assume $\SSS \cap \Sigma = \emptyset$), and thus we use the alphabet $\Sigma \cup \SSS$.

For $g \in G_i$, the automorphism $\hat g$ applies the natural conveyor belt action of $g$ in segments of type~$i$ (through decoding the words in $W_i$ to elements of $B$). If there is an element $s \in \SSS$ to the left of the segment, and a segment of type~$j$ immediately to the left of $s$, where $j \neq i, \{i,j\} \notin E$, then additionally permute the rightmost $W_j$-word of the segment of type~$j$, as we did in the finite case. Since the graph is locally cofinite, the function is continuous, and thus we obtain an automorphism action. The proof that this gives a faithful action of the graph product is analogous to the finite case.

\vspace{0.2cm}

\textbf{Arbitrary $\Gamma$.} Now, consider an arbitrary countable graph $\Gamma$. In this case, we modify the previous construction further: We set $\SSS = B \times B^B \times B$. The set of positions where letters from $\SSS$ appear will stay fixed in the construction (as they did before). After an element $\hat g$ is applied (for any $g \in G_i$), if the rightmost $B$-component of $s \in \SSS$ is changed from $d$ to $d'$, as a side-effect we will always change $(a, b)$ in the leftmost $B$-component to $(a, b) + s'(d') - s'(d)$, where $s'$ is the $B^B$-component of $s$. (This description does not make sense if $\hat g$ modifies both the leftmost and rightmost $B$-component, but we make sure this never happens.)

The idea is now that instead of segments of type $i$ actually knowing whether they should interact with a segment $j$, we will have a logic for whether we include the $B$-symbol of an $\SSS$ in the present segment, so that interaction between a type-$i$ segment and a type-$j$ segment only happens when \emph{both} include the $B$-symbol. The segment of the ``larger type'' will prevent the interaction if the groups are supposed to commute, by not including the $B$-symbol in its conveyor belt, thus every segment only has to know about segments smaller than it, which leads to a finite radius for the rule.

If $s \in \SSS$ has a segment of type~$i$ to the left of it, and a segment of type~$j$ to the right of it, then we ``include'' the rightmost $B$-component of $s$ to the segment on the right (i.e.\ we think of it as joined in the beginning of the conveyor belt coded by the $W_j^*$-word) if
\[ i \neq j \wedge (i > j \vee \{i,j\} \notin E). \]
We include the leftmost $B$-component of $s$ to the segment on the left if
\[ i \neq j \wedge (j > i \vee \{i,j\} \notin E). \]
When $s$ has a segment on only one side, the $B$-symbol on the side of the segment is always included in the segment.

The idea of this is that ``by default'', the $B$-symbols in a symbol $s$ are joined to the conveyor belts on each side. However, under some circumstances, we do not include them. The side-effects of applying elements of type $G_j$ on segments of type $i$ will of course only be visible if we include the symbols on both sides, and this is why it does not matter that a single $i$-type will have infinite many interactions -- when interacting with groups $G_j$ for $j > i$, the $B$-symbol will be included by default on the side of the type-$i$ segment, and the rule for $j$ on the other side will exclude it if necessary.

Now again for $g \in G_i$, the map $\hat g$ simply applies $g$ in segments of type~$i$ (with the rule described above determining which $B$-symbols are part of the conveyor belt), and the $S$-symbols themselves take care of the side-effects as described above. To see that this is an automorphism, observe that a local rule can tell whether $B$-components of possible neighboring $\SSS$-symbols belong to the segment since when $g \in G_i$, this requires only knowing the sets $W_k$ up to $k \leq i$.

Specifically, the only case when the leftmost $B$-symbol of an $\SSS$-symbol is \emph{not} included in a type-$i$ segment is that there is a segment of some type $j$ to the right, and either $i = j$, or $j < i$ and $\{i, j\} \in E$, and similar logic applies for inclusion on the left side.

Note that as required for the side-effect logic to make sense, the rule never modifies both $B$-components of a symbol $s \in \SSS$, because that would mean the segments on both sides are of type~$i$, and in this case neither $B$-component would be included in the segments on its side.

The proof from the finite (and cofinite) case now goes through with minor modifications, although the inclusion of the bordermost $B$-symbols in the $\SSS$-symbol effectively puts a lower bound on the length of some of the segments in the proof of faithfulness of the action. It suffices to make the assumption that every nontrivial $g \in G_i$ acts on $B^\Z$ with minimal radius at least $10$, say; this is again possible because by \cite{KiRo90} there exists a self-embedding of $\Aut(\Sigma^\Z)$ such that all nontrivial automorphisms in the image have minimal radius at least $10$.

By the assumption on the minimal radius of the maps $g \in G_i$, this gives a correct embedding of the graph product with an analogous proof as in the above cases.
\end{proof}

As noted in the introduction, the above proof does not ``use'' the normal form theorem \cite{Gr90} for graph products (though of course this normal form theorem is why we knew we can make reduced words act non-trivially), so as a side-effect we obtain a proof of the usual normal form theorem in the case where node groups act by automorphisms on a full shift (which admittedly is more complicated than the usual one).

Sometimes it is convenient to have the groups $G_i$ act on some full shift $\Sigma^\Z$ other than $\{0,1\}^\Z$, and it is clear that one can modify the construction to use $\Sigma^2$ instead of $B = \{0,1\}^2$. The set $B^B$ is simply replaced with $({\Sigma^2})^{\Sigma^2}$. We use this in the proof of Lemma~\ref{lem:FreeProduct}. In the case of finite graph products, one can even use a different alphabet in each segment (replacing $B^B$ by the set of functions from one square alphabet to another). 

Besides graph products, another generalization of the free product is the free product with amalgamation. Here, there is also a simple normal form, and in fact using the same ``segments of different types'' construction it seems plausible that one can prove at least some restricted closure results for this operation.

\begin{question}
If $G, H \in \mathcal{G}$, when is $G *_K H \in \mathcal{G}$?
\end{question}

If we take all the node groups to be $\Z_2$ (resp.\ $\Z$), but instead of commutation relations we add relations of the form $(st)^m = e$, we obtain the family of Coxeter (resp.\ Artin) groups. Indeed graph groups are also called \emph{right-angled Artin groups} (and there is also a notion of \emph{right-angled Coxeter group}, which is a graph product of $\Z_2$s).

\begin{question}
Which finitely-generated Coxeter (resp.\ Artin) groups are in $\mathcal{G}$?
\end{question}

By a Theorem of Tits, all f.g.\ Coxeter groups are linear, and thus residually finite \cite{CoLoRe98,Br89}. They are automatic \cite{BrHo93} so their word problem is decidable in polynomial time. These facts show that we cannot show (with present tools) that these groups do not all belong to $\mathcal{G}$. A general reference for these groups is \cite{Da12}.

\section{$A$-cancellation and unique visits}

In this section, we define various technical properties of actions of a subgroup of an automorphism group of a full shift, which try to formalize the idea of an action by cellular automata where the nontriviality of any nontrivial element ``can be seen in a single cell''.

These will be used in the following section to attack wreath products. Specifically, we define actions without $A$-cancellation and actions with unique visits (which are dynamical concepts of possible general interest), and property~P (which seems to be quite specific to our ``automorphisms on full shifts'' setting).

Property~P implies unique visits (at least up to passing to a diagonal action), which in turn implies non-$A$-cancellation. Not having $A$-cancellation (and thus either of the other properties) is sufficient for embeddability of wreath products $A \wr G$. On the other hand, we show that unique visits and property~P are relatively robust concepts with good closure properties, and we exhibit examples of groups acting with these properties.

%Now we move on to generalizations of the lamplighter group $\Z_2 \wr \Z$. Recall that we restrict our systems to be discrete group actions on compact metrizable zero-dimensional spaces. This includes subshifts and the natural actions of automorphism groups of subshifts. 

\begin{definition}
\label{def:Aithful}
Let $X$ be a $G$-system and let $A$ be a finite abelian group. If $B$ is a finite abelian group and $\theta : X \to \Hom(A,B)$ is continuous, then $(G, X)$ has \emph{$A$-cancellation with respect to $(\theta, B)$} if there exists a finite-support map $f : G \to A$ such that
\[ (\forall x \in X: \sum_{g \in G} \theta(gx)(f(g)) = 0) \implies f = 0. \]
We say $(G, X)$ \emph{has $A$-cancellation} if it has $A$-cancellation with respect to every $(\theta, B)$.
%We say the system $X$ \emph{has no $A$-cancellation} if there exists a finite abelian group $B$ and a continuous function $\theta : X \to \Hom(A,B)$ such that for any finite support map $f : G \to A$ we have
%\[ (\forall x \in X: \sum_{g \in G} \theta(gx)(f(g)) = 0) \implies f = 0. \]
%If the system does not have $A$-cancellation, we also say it \emph{has no $A$-cancellation}.
\end{definition}

%The complement of this (which means cancellation for all choices of $\theta$) is referred to as $A$-cancellation, to avoid a double negation. We sometimes say a system has $A$-cancellation with respect to a particular homomorphism $\theta$ .

Here we consider $\Hom(A,B)$ with the discrete topology (as it is a finite set). While the function notation is convenient, a continuous function from a zero-dimensional compact metric space to a finite discrete set $S$ just means a finite clopen partition of the space where the partition elements are indexed by $S$. Of course the definition could be applied in more general contexts, but we have not explored this, and $A$ will always be finite abelian for us.

Algebraically, when $A$ is abelian (and written additively) maps $f : G \to A$ with finite support form an abelian group under elementwise sum. Analogously to group rings, we write this group as $A[G]$ (but there is typically no product) and write elements as (essentially) finite sums $\sum_{g \in G} f(g) \cdot g$. We can define a map $\phi = \phi_{\theta,B} : A[G] \to B^X$ by $f \mapsto (x \mapsto \sum_g \theta(gx)(f(g)))$, and this is clearly a homomorphism (for any choice of $\theta$). Having no $A$-cancellation means that this homomorphism is injective for some choice of $B$ and $\theta$.

\begin{lemma}
\label{lem:BAd}
We can always pick $B = A^d$ for some $d$ in the definition of $A$-cancellation.
\end{lemma}

\begin{proof}
The map $\theta$ is determined by some finite clopen partition $P_1 \sqcup P_2 \sqcup \cdots \sqcup P_d$ of $X$. We can factor $\theta$ as the pointwise composition $x \mapsto \theta''(x) \circ \theta'(x)$, where $\theta'|_{P_i} : P_i \to \Hom(A, A^d)$ is the constant map with image
\begin{equation} \label{eq:ith} a \mapsto (0,0,...,0,\underset{i\mathrm{th\;position}}{a},0,...0) \end{equation}
and $\theta'' : X \to \Hom(A^d, B)$ is the constant map with image the homomorphism $\alpha : (0,0,...,0,a,0,...0) \mapsto \theta(x)(a)$ for any $x \in P_i$, where $a$ appears in the $i$th position. Since $\theta''$ is a constant map and $\alpha$ is a homomorphism, a short calculation shows
\[ \phi_{\theta,B}(f)(x) = \alpha(\phi_{\theta',A^d}(f)(x)), \]
and since we are aiming for injectivity, it is safe to drop $\theta''$ and replace $\theta$ with $\theta'$.
\end{proof}

The proof shows more: we can pick the homomorphism to be the one described in~\eqref{eq:ith}, with respect to some finite clopen partition.

\begin{lemma}
Let $X$ be a $G$-system and let $A$ be a finite nontrivial abelian group. If the action does not have $A$-cancellation then there is a faithful subshift factor (in particular, not having $A$-cancellation implies faithfulness).
\end{lemma}

Here, recall that group actions on compact metrizable zero-dimensional spaces are inverse limits of expansive ones, namely they are inverse limits of their subshift factors obtained by recording only the current partition element visited along an orbit, with respect to a finite partition.

A faithful subshift factor implies faithfulness, but not vice versa. For example the natural action of a residually finite group on its profinite completion is faithful but its subshift factors are finite.

\begin{proof}
We prove the contrapositive, so suppose no subshift factor is faithful. Let $B$ be an abelian group and let $\theta : X \to \Hom(A,B)$ be continuous. Then $\theta$ factors through a subshift $Y$ as $\theta = \theta'' \circ \theta'$ with $\theta' : X \to Y$, $\theta'' : Y \to \Hom(A, B)$. Since $(G, Y)$ is not faithful, there exists $k \in G$ such that $ky = y$ for all $y \in Y$. Let $f = -a \cdot k + a \cdot e_G$. We have
\begin{align*}
\sum_{g \in G} \theta(gx)(f(g)) &= \sum_{g \in G} \theta''(\theta'(gx))(f(g)) \\
&= \theta''(\theta'(kx))(f(k)) + \theta''(\theta'(x))(f(e)) \\
&= \theta''(\theta'(x))(-a) + \theta''(\theta'(x))(a) = 0,
\end{align*}
where $\theta'(kx) = k\theta'(x) = \theta'(x)$ by the assumption on $k$. Thus, the map $\phi_{\theta, B} : A[G] \to B^X$ is not injective for any $(\theta, B)$, showing that the system has $A$-cancellation.
\end{proof}

\begin{lemma}
Let $G = \Z$, let $X$ be a $G$-system, and let $A$ be a finite nontrivial abelian group. If $(G,X)$ has a faithful subshift factor then it is does not have $A$-cancellation.
\end{lemma}

\begin{proof}
We prove the contrapositive, so suppose the action is has $A$-cancellation. Consider any clopen partition $X = P_1 \sqcup \cdots \sqcup P_n$, and for $i \in \{1,\dots,n\}$ define $\theta : X \to \End(A)$ by $\theta(x) = \id$ for $x \in P_i$, $\theta(x) = 0$ otherwise. Since the action has $A$-cancellation, there exists nonzero $f \in A[G]$ such that
\[ \sum_{g \in G} \theta(gx)(f(g)) = 0 \]
for all $x \in X$.

Let $g' = \max \{g \in \Z \;|\; f(g) \neq 0_A\}$ in the usual ordering of $\Z$. Then for all $x \in X$ we have $\sum_{g < g'} \theta(gx)(f(g)) = -\theta(g' x) (f(g'))$ for all $x$. We have to have $g' x \in P_i$ whenever $\sum_{g < g'} \theta(gx)(f(g)) \neq 0$ (though this may not be sufficient), and have to have $g' x \notin P_i$ whenever $\sum_{g < g'} \theta(gx)(f(g)) = 0$, so whether $g' x \in P_i$ can be deduced from $\sum_{g < g'} \theta(gx)(f(g))$, and thus from the set of $g < g'$ such that $f(g) \neq 0$ and $g x \in P_i$.

This is easily seen to imply that the set of all $g \in G$ such that $g \cdot x \in P_i$ forms a finite union of arithmetic progressions with a bounded period (over all of $X$). Since this happens for all $i = 1, \dots, n$, the subshift factor given by the partition is uniformly periodic, thus not faithful.
\end{proof}

In particular, for $\Z$-subshifts, faithfulness (which is equivalent to infiniteness) is equivalent to not having $A$-cancellation for any nontrivial $A$.

Also, it is shown in \cite{MeSa19} that for $\Z$-actions commuting with $G$-subshifts on any group $G$ (i.e.\ $\Z$-actions of automorphisms of subshifts), faithfulness of the $\Z$-action is equivalent to faithfulness of the subshift factor given by the partition giving $G$-expansivity. We conclude that also for $\Z$-actions that are automorphisms of a $G$-subshift, faithfulness is equivalent to not having $A$-cancellation for any nontrivial $A$.

\begin{question}
For which pairs $(G,A)$ does there exist a faithful $G$-subshift which has $A$-cancellation?
\end{question}

%Apart from $G = \Z$, we do not know the answer for any $(G,A)$.
% The natural action of the lamplighter group is a subshift, which has \Z_2-cancellation. Consider namely any partition... So it's putting on different \Z_2-tracks things depending on what it sees.
We do not know any faithful expansive actions of $\Z^d$ which have $A$-cancellation. We next show that the Ledrappier subshift does not have $A$-cancellation for $A = \Z_2$, even though (by definition) it admits a cancelling pattern over $\Z_2$. (In fact, it can be shown to not have $A$-cancellation for any nontrivial $A$.)

\begin{example}
\label{ex:Ledra}
Consider the Ledrappier subshift
\[ X = \{x \in \Z_2^{\Z^2} \;|\; \forall v \in \Z^2: x_v + x_{v+(0,1)} + x_{v+(1,0)} = 0\}, \]
and let $A = \Z_2$. Let $\theta : X \to \Hom(\Z_2,B)$ be any continuous function for $B$ an abelian group. We may assume $B = \Z_2^d$ for some $d$, and then $|\Hom(\Z_2, B)| = 2^d$. The map $\theta$ is then determined by $d$ and by a partition $(C_v)_{v \in B}$.

If the partition depends only on the symbol at $(0,0)$, then the system always has $\Z_2$-cancellation with respect to $\theta$: take
\[ f = 1_{\Z_2}\cdot(0,0) + 1_{\Z_2}\cdot(1,0) + 1_{\Z_2}\cdot(0,2) + 1_{\Z_2}\cdot(1,1), \]
(where $1_{\Z_2}$ is the generator of $\Z_2$) and consider the sum
\[ \sum_{v \in \Z^2} \theta(vx)(f(v)). \]
By assumption, $\theta(vx)$ only depends on $x_v$, and contributes a particular vector of $B$ depending on this value. In any single coordinate of $B$, we thus either add $1$ in any case (in which case the sum is $0$ since $f$ has support of size four), never add $1$ (a trivial case), or we add $1$ when $x_v = a$ for a particular $a \in \Z_2$. This amounts to counting the parity of the number of $0$s or $1$s in $x|_{\{(0,0),(1,0),(0,2),(1,1)\}}$, and a short calculation shows that this is always even.

Setting $d = 1$, $C = \{x \in X \;|\; (x_{(0,0)},x_{(-1,0)}) = (0,1)\}$, and letting $\theta|_C = \id, \theta|_{X \setminus C} = 0$, on the other hand, proves non-$\Z_2$-cancellation:
consider any nonzero $f \in \Z_2[\Z^2]$. The sum $\sum_{v \in \Z^2} \theta(vx)(f(v))$ now amounts to counting (modulo $2$) how many times $C$ is seen on the support of $f$, i.e.\ how many $0$s there are on the support of $f$, such that the symbol on the left is $1$.

Let $i \in \Z$ be the leftmost x-coordinate that appears in the support of $f$. There is a configuration in the Ledrappier subshift such that on the $i$th column, there is exactly one occurrence of $0$, and we can align it with one of the elements in the support of $f$ to obtain a configuration $x$. There is also a configuration $y \in X$ satisfying $y_{(i-1,j)} = 1$ for all $j \in \Z$, and any such configuration satisfies $y_{(i+k,j)} = 0$ for all $k \geq 0, j \in \Z$. Clearly the number of times $C$ is entered by $x$ and $x+y$ in the support of $f$ differs by exactly one, so one of these numbers is odd, and this configuration proves non-$\Z_2$-cancellation. \qee
\end{example}

Most of our examples of actions with non-$A$-cancellation, in particular all our cellular automata actions, come from the following stronger property.\footnote{It is possible that also the Ledrappier subshift has this property, but we have no proof.}

\begin{definition}
\label{def:StronglyFaithful}
An action $(G,X)$ has \emph{unique visits} if there is a clopen set $C \subset X$ such that for all $\emptyset \neq F \Subset G$ there exists $x \in X$ such that $\exists! g \in F: gx \in C$.
\end{definition}

\begin{lemma}
Every action with unique visits is faithful and does not have $A$-cancellation for any non-trivial $A$.
\end{lemma}

\begin{proof}
Assume a $G$-system $X$ has unique visits. Since non-$A$-cancellation implies faithfulness (and even a faithful subshift factor), it suffices to prove the latter claim, but we give a direct proof of faithfulness: let $g \in G \setminus \{e\}$ and take $F = \{e, g\}$. By unique visits there exists $x$ such that exactly one of $x, gx$ is in $C$. In particular $x \neq gx$ and the action is faithful.

To prove non-$A$-cancellation, let $A$ be any finite abelian group. Define $\theta : X \to \End(A)$ by $\theta^{-1}(\id_A) = C$, and $\theta(x) = 0$ for $x \notin C$. Suppose $f \in A[G]$ and that $\sum_{g \in G} \theta(gx)(f(g)) = 0$ for all $x \in X$. Let $F \Subset G$ be the support of $f$. If $F \neq \emptyset$, let $x \in X$ be given by unique visits, so there is a unique $g' \in F$ with $g'x \in C$. By the choice of $\theta$ we then have $\sum_{g \in G} \theta(gx)(f(g)) = \theta(g'x)(f(g')) = f(g') \neq 0$, a contradiction. Thus $F = \emptyset$, i.e.\ $f = 0$, thus the action has no $A$-cancellation.
\end{proof}

The following are proved in a straightforward fashion from the definitions.

\begin{lemma}
\label{lem:Subgroup}
If $(G, X)$ has unique visits (resp.\ non-$A$-cancellation) then so does $(H, X)$ for any $H \leq G$.
\end{lemma}

\begin{lemma}
\label{lem:Supersystem}
If $(G, X)$ has unique visits (resp.\ non-$A$-cancellation) then so does $(G, Y)$ for any $Y \supset X$.
\end{lemma}

We sandwich unique visits between two other properties, which may clarify it. A \emph{non-recurrent point} in $X$ is $x \in X$ such that for some open set $U \ni x$, we have $gx \in U \implies g = e_G$. A \emph{horoball} in a finitely-generated group $G$ is a limit of balls with radius tending to infinity, with respect to a fixed generating set (see \cite{EpMe20}; here we include the degenerate horoball $G$, and one should take limits with right translates of balls $B_r g$).

\begin{lemma}
\label{lem:SFImplications}
If $(G,X)$ has a non-recurrent point, it has unique visits.
If $(G,X)$ has unique visits and $G$ is generated by the finite set $S$, then there exists a clopen set $C$, an $S$-horoball $H$ and $x \in X$ such that $hx \in C$ for a unique element $h \in H$.
\end{lemma}

\begin{proof}
For the first claim, suppose %$Y \subset X$ has a sunny-side-up factor $\phi : Y \to X_{\leq 1}$. Observe that $[1] = \{x \in X_{\leq 1} \;|\; x_e = 1\}$ is clopen in $X_{\leq 1}$ so $\phi^{-1}([1])$ is clopen in $Y$, so by basic topology there is a clopen set $C \subset X$ such that $C \cap Y = \phi^{-1}([1])$. Let $y \in Y$ be any element of $C \cap Y$. Then we can use $C$ as the clopen set, and translates of the point $y$ satisfy the definition of stro faithfulness for any finite set $F$.
$(G, X)$ has a non-recurrent point $x \in X$. Then there is also a clopen set $C \ni x$ such that $gx \in C \implies g = e_G$. Using this clopen set, translates of $x$ can be used to satisfy the definition of unique visits.

For the latter claim, under unique visits, for any finite $F \Subset G$ there exist $x \in G$ and $g \in F$ such that
\[ \{h | hx \in C\} \cap F = \{g\} \]
This implies $\{hg^{-1} | hx \in C\} \cap F g^{-1} = \{e\}$ implies
\[ \{h | h g x \in C\} \cap F g^{-1} = \{e\} \]
and thus letting $y(F) = g x$ and $t(F) = F g^{-1}$ we have $y(F) \in C$ and $\forall h \in t(F) \setminus \{e\}: hy(F) \notin C$.

Let $P$ be the set of pairs $(H, z)$ such that $e \in H$ and 
\[ \forall h \in H: (hz \in C \iff h = e). \]
Clearly this set is closed in $2^G \times X$ since $C$ is clopen. In particular applying the observation of the previous paragraph to balls with respect to the generating set $S$, one obtains an $S$-horoball $H$ and $z \in C$ such that the $H$-orbit of $z$ does not revisit $C$.
\end{proof}

%We note a funny way to rephrase the condition of having a wandering point: If $G$ is an infinite group, let $X_{\leq 1} = \{x \in \{0,1\}^G \;|\; \sum x \leq 1\}$ be the \emph{sunny-side-up} subshift on $G$. Say a system is a \emph{subquotient} of another if it is a factor of a subsystem. Then it can be checked that having a sunny-side-up subshift as a subquotient is equivalent to having a point that is isolated in its orbit.

\begin{remark}
Of course one can apply the argument we used to get horoballs to finite sets other than balls.
In the case of $\Z^2$, in a precise sense an ``optimal'' sequence of sets to apply it to are discretized $\ell^2$-balls whose boundaries eventually contain arithmetic progressions in every rational direction, namely this forces $t$ to give a set $H \ni \{(0,0)\}$ in the limit which contains the predecessor set $N$ of $(0,0)$ in some translation-invariant total order of $\Z^2$ -- this is the best we can do since $t$ could plausibly always force translates inside any such set $N$. \qee
\end{remark}

\begin{remark}
Having a non-recurrent point can be stated in many ways. First, it is immeidately equivalent to having a point that is isolated in its own orbit. Second, if $G$ is an infinite group, let $X_{\leq 1} = \{x \in \{0,1\}^G \;|\; \sum x \leq 1\}$ be the \emph{sunny-side-up} subshift on $G$. Say a system is a \emph{subquotient} of another if it is a factor of a subsystem. Now, it is not difficult to show that having a sunny-side-up subshift as a subquotient is equivalent to having a point that is isolated in its orbit.
\end{remark}

\begin{example}
There is a unique visits subshift where no point is non-recurrent. For example the Cantor's dust $\Z^2$ subshift \cite[Figure~1c]{Sa20} has this property. On the other hand, consider any $\Z^2$-subshift obtained as the orbit closure of a discretization of a rational line in direction $\vec v$ (see \cite[Figure~1b]{Sa20}, but consider a rational continuation). Such a subshift does not have unique visits (as it is not even faithful), but for any generating set whose convex hull has no edge in direction $\vec v$ (up to orientation), there are points that visit $[1]$ (the clopen set of drawings that hit the origin) just once in a horoball. \qee
\end{example}

\section{Actions with unique visits by cellular automata}
\label{sec:SFActions}

\subsection{Robustness properties}
% The following lemma is mainly of interest to readers interested in sofic shifts \cite{LiMa95}, so we omit a detailed discussion of their structure. However, 

Before giving examples of interesting groups acting with unique visits, we prove some closure and robustness properties. First, we observe that just like the definition of $\mathcal{G}$ does not depend on the alphabet, also groups acting with unique visits or without $A$-cancellation do not depend on the alphabet. We prove this more generally for sofic shifts. In brief, sofic shifts on the group $\Z$ are (up to isomorphism) sets of edge-labelings of bi-infinite paths in finite edge-labeled directed graphs. More information on sofic shifts can be found in \cite{LiMa95}.

\begin{lemma}
\label{lem:FullShiftThenSofic}
If $G$ acts with unique visits (resp.\ without $A$-cancellation) by automorphisms on some full shift $(\sigma, \Sigma^\Z)$, then it admits an action with unique visits (resp.\ non-$A$-cancellation) by automorphisms on any sofic $\Z$-shift with uncountably many points.
\end{lemma}

\begin{proof}
As we explained in the proof of Theorem~\ref{thm:GraphProducts}, conveyor belts allow us to simulate an action by automorphisms of $\Delta^\Z$, by automorphisms of $\Sigma^\Z$, by picking an unbordered set of words over $\Sigma$, all of the same length, which are in bijection with $\Delta^2$, connecting the simulated symbols of $\Delta$ into conveyor belts, and applying the simulated automorphisms on these belts. This construction is done more generally for a class of SFTs in \cite{KiRo90}, and generalized to all uncountable sofic shifts in \cite{Sa18d}. The idea is exactly the same in each case, so we concentrate on the case of an embedding into the automorphism group of a full shift $\Sigma^\Z$.

Now suppose we have an action by some group $G$ of automorphisms of $\Delta^\Z$ which has unique visits. Then we claim that automatically any conveyor belt simulation of it by $H \cong G$ also has unique visits. Namely if $C$ is a clopen set for the original action proving unique visits, then we can describe a clopen set $C'$ proving unique visits as follows: Given $x \in \Sigma^\Z$, determine the simulated configuration on conveyor belts near the origin. Put $x \in C'$ if the simulated configuration on the top track is in $C$. This gives a finitary rule for determining whether $x \in C'$, so $C'$ itself is clopen. Considering configurations where the simulated configurations are $\Z$-shaped (i.e.\ the conveyor belts do not wrap around), the $G$-action having unique visits implies the same property for $H$.

For non-$A$-cancellation, the argument is completely analogous.
\end{proof}

More abstractly, one may observe that there is a closed set in $\Sigma^\Z$ where we fix the shape of the conveyor belt containing the origin so that it does not wrap, have no other conveyor belts, and fix the data outside of the conveyor belt. Then the dynamics on this closed set is conjugate to the action of the automorphism group of $\Delta^\Z$, and we inherit the unique visits property and $A$-cancellation from any subsystem with this property, by Lemma~\ref{lem:Supersystem}.

%The following closure property is proved mainly to show some further robustness of the stro faithfulness notion, but does not have a direct use.

\begin{lemma}
\label{lem:Virtually}
If $G$ acts by automorphisms on some full shift $(\sigma, \Sigma^\Z)$ with unique visits, and $H$ is commensurable to $G$, then $H$ admits an action with the same properties.
\end{lemma}

\begin{proof}
By Lemma~\ref{lem:Subgroup}, it is enough to show that the class of groups acting with unique visits by full shift automorphisms is closed under passing to finite index supergroups, so suppose $G \leq H$ is of finite index and $G$ admits a such an action. Suppose first $G$ is infinite.

The result is a straightforward from the proof of closure under finite extensions in \cite{KiRo90}. We briefly recall this idea. Recall that if $[H : G] = n$, \cite{KiRo90} constructs the induced representation of $H$ on $(\Sigma^n)^\Z$ from a set of left coset representatives $H = h_1G \cup \cdots \cup h_kG$ where $h_1 = e_H$, by the intuition that a configuration $x \in (\Sigma^n)^\Z$ can be seen as $(x_1, \ldots, x_k) \in (\Sigma^\Z)^n$ and writing this as $(h_1 x_1, \ldots, h_k x_k)$, there is an obvious formula for the action of $h \in H$ on such a tuple: simply write the set $\{h h_1 x_1, \ldots, h h_k x_k\}$ in the form $\{h_1 y_1, \ldots, h_k y_k\}$ in the natural way (noting that for all $i$, we have $h h_i \in h_j G$ for some $j$, and $i \mapsto j$ is a permutation), and then reorder the tuple.

Let $C$ be the clopen set proving unique visits for the $G$-action, and let $C'$ to be the clopen set of configurations where the coordinate corresponding to the trivial coset contains a configuration from $C$.

Now consider nonempty $F \Subset H$, w.l.o.g.\ suppose $e_G \in F$ so $F \cap G$ is a nonempty subset of $G$ and apply unique visits to obtain a configuration $x$ such that there is a unique element $h \in F \cap G$ such that $h x \in C$. Put the configuration $x$ in the coordinate corresponding to the trivial coset, and in other coordinates put a configuration $y$ whose $G$-orbit does not enter $C$ (which is possible when $G$ is infinite by e.g.\ the existence of a sunny-side-up subquotient). Let $z \in (\Sigma^n)^\Z$ be the resulting configuration.

Now, observe that the $H$-orbit of $z$ contains only configurations where the coordinates contain elements from $Gx$ and $Gy$, and the coordinate of $gz$ corresponding to the trivial coset contains an element of $Gx$ if and only if $g \in G$. Now $hz$ is in $C'$ by the definition of the induced action. On the other hand, $gz \in C'$ implies that the trivial coset coordinate of $gz$ contains an element in $C$, thus an element of $Gx$, thus $g \in G$. But $g \in F \cap G$ implies $g = h$, thus $h$ is the only element of $F$ mapping $z$ into $C'$.

To cover finite extensions of finite groups $G$, it is necessary and sufficient to prove that all finite groups act with unique visits. The cellwise application of the regular representation of $G$ gives an action on $G^\Z$ by automorphisms with unique visits.
\end{proof}

\subsection{Direct products}

%It seems that typically in concrete cases of applying the the graph product construction of Theorem~\ref{thm:GraphProducts} to free products, we obtain a strly faithful action, and we do not have examples where an aityh action is not obtained. Nevertheless, we do not know whether the set of groups acting ly faithfully (or $Arealy) by full shift automorphisms is closed under free product. We also do not know if the set of groups acting stry faithfully by automorphisms of a full shift is closed under $G \mapsto \Z_2 \wr G$, though in this case we know that the construction given in the following section usually does not give an $Ara action.

\begin{lemma}
\label{lem:Product}
The class of groups acting with unique visits by automorphisms on a full shift is closed under finite direct products.
\end{lemma}

\begin{proof}
%The product action of stryly faithful actions is obviously tyfly faithful, and can be seen as an automorphism action on a full shift on a product alphabet.
Let $G_1, \ldots, G_n$ be groups acting  with unique visits by automorphisms on full shifts with disjoint alphabets $\Sigma_i$. Let $C_i$ be the respective clopen sets proving unique visits. Consider the natural action of the group $G = \bigoplus_i G_i$ on $\Delta = \prod_i \Sigma_i$ with $G_i$ acting on the $\Sigma_i^\Z$-component. We prove unique visits with $C = \prod_i C_i$.

If $F \Subset G$, consider the natural projection $\pi_1(F) = F_1$ of $F$ into $G_1$. This is a nonempty finite subset of $G_1$, so by unique visits of the action there exists $g_1 \in F_1$ and $x_1$ such that $g_1 x_1 \in C_1$ but $h_1 x_1 \notin C_1$ for $h_1 \in F_1 \setminus \{g_1\}$. Note that if $h \in F$ does not have $\pi_1$-projection $g_1$, then $h (x_1, y_2, \ldots, y_n) \notin C$ for any choice of $y_i$.

Now repeat the previous argument for $F \cap \pi_1^{-1}(g_1)$ and projection $\pi_2$, to obtain $g_2$ such that $\pi_1 \times \pi_2 (F) \ni (g_1, g_2)$ and $(g_1, g_2)(x_1, x_2) \in C_1 \times C_2$ but $h (x_1, x_2, y_3, \ldots, y_n) \notin C$ if $\pi_1 \times \pi_2(h) \neq (x_1, x_2)$. Repeating this argument, we finally obtain $g = (g_1, \ldots, g_n) \in F$ such that $g(x_1, x_2, \ldots, x_n) \in C$ but $h(x_1, x_2, \ldots, x_n) \notin C$ for $h \in F \setminus \{g\}$.
\end{proof}

%It is less obvious that the product action of a finite family of $Asre actions is $Arwa, as the assumption only tells us something about polynomials with a product decomposition.

We do not know whether either class (unique visits or non-$A$-cancelling cellular automata actions) are closed under infinite direct sums, indeed for infinite sums we used a different construction in the proof of Theorem~\ref{thm:GraphProducts}, and the following remark outlines the problem with this approach, at least for functions that look at only one coordinate (and non-$A$-cancellation fails even for product polynomials).

\begin{remark}
The direct sum case of Theorem~\ref{thm:GraphProducts} does not typically give actions with unique visits or non-$A$-cancellation, at least if the clopen set $C$ (resp.\ function $\theta$) looks at just one coordinate. Perform the construction for $G \times G'$. It turns out the action has $A$-cancellation no matter how $G$ and $G'$ act. To see this, take $f = (e_G - g) (e_{G'} - g') \in \Z_2[G \times H]$.

Here we suggestively use additive notation and standard conventions and notation for the group ring structure, but strictly speaking we mean the element
\[ f = 1_{\Z_2} \cdot (e,e) + 1_{\Z_2} \cdot (g,e) + 1_{\Z_2} \cdot (e,g') + 1_{\Z_2} \cdot (g,g'). \]
Observe now that if the segment at the origin of a configuration $x$ is ``of type~$G$'' (the case $G'$ being symmetric), then we have
\begin{align*}
\sum_{(k,k') \in G \times G'} \theta((k,k')x)(f(k,k')) = \; & \theta(x)(f(e,e)) - \theta((g,e)x)(f(g,e)) \; + \\
& \theta((g,g')x)(f(g,g')) - \theta((e,g')x)(f(e,g')) = 0,
\end{align*}
if $\theta$ only looks at the central coordinate: we have
\[ \theta(x)(f(e,e)) = \theta((e,g')x)(f(e,g')) \]
and 
\[ \theta((g,e)x)(f(g,e)) = \theta((g,g')x)(f(g,g')), \]
since $f((e,e)) = f((e,g')), f((g,e)) = f((g,g')$ by our choice of $f$ and since the action of $(e,g')$ is not visible at the origin. Similarly, unique visits are contradicted for any $C$ that only looks at the central coordinate. \qee
\end{remark}

\subsection{Property~P and free products}

Next we construct actions with unique visits for some free products. We do not know whether groups acting by automorphisms with unique visits are closed under free product (in fact, we do not know whether $\Aut(\Sigma^\Z)$ admits an action by automorphisms on $\Sigma^\Z$ with unique visits), but we define a technical strengthening of this property such that groups acting with this property are closed under free product.

To motivate this definition, we start by a simple remark. Consider $F \Subset \Aut(\Sigma^\Z)$ a finite set of automorphisms. Then the group $\langle F \rangle$ also acts naturally on $(\Sigma^2)^\Z \cong (\Sigma^\Z)^2$ by the diagonal action. Now if $g \in F$ has minimal radius $r$ which is maximal among elements of $F$, then we can find two points $x, y$ such that $x_i \neq y_i$ with $|i| = r$, and $x|_{\{-i+1, \ldots, i-1\}} = y|_{\{-i+1, \ldots, i-1\}}$, and $g(x)_0 \neq g(y)_0$. Clearly, for any $h \in F \setminus \{g\}$ with minimal radius less than $r$, we will have $h(x)_0 = h(y)_0$.

For example, one can show using this that the diagonal action of the shift map has unique visits: In this case $F \Subset \langle \sigma \rangle$ and there will be at most two $g, g' \in F$ sharing the maximal minimal radius. It suffices to show that we can find a pair of configurations $x, y$ where a difference propagates at maximal rate for $g$, but the same does not happen with $g'$, and this is clear because these maps move the difference in different directions. (Of course, for $\Z$ this is not so useful, as an equally easy if not easier proof shows that the standard action of the shift already has unique visits.)

Note that it actually suffices that we can find $g \in G$ such that we can find \emph{some pair} $x, y$ of configurations such that when $g$ is applied to the pair, a difference $x_i \neq y_i$ ``travels'' faster into an area $I \subset \Z$ where $x|_I = y|_I$, than in the application of any other element of $G$. This sufficient condition is more or less formalized by the following property, and will turn out sufficient for closure under free products. %if in any nonempty finite set $F \Subset G$ we can find an element $g$ which propagates a difference between some two configurations at a maximal rate,

%\begin{definition}
%Let $G$ be a group of cellular automata. We say $G$ has \emph{property~P} if for all $g \in G \setminus \{e_G\}$ acting with minimal radius $r$, the following holds: there exist $u \in \Sigma^{2r}$ and $a, b, c, d \in \Sigma$ such that $g(auc) \neq g(bud)$, while for all $h \in G \setminus \{g\}$ we have $h((auc)^\Z)|_{\{r,r+1\}} = h((bud)^\Z)|_{\{r,r+1\}}$.
%\end{definition}

%In the above definition, $auc$ is a word of length $2r+2$, and $g(auc)$ and $g(bud)$ refer to the application of the local rule of $g$ to the prefix and suffix of length $2r+1$ to obtain a word of length $2$. The definition is not particularly natural from a dynamical point of view, it is simply what we encounter in the following proof, and is sufficient for our purposes

\begin{definition}
\label{def:P}
Let $G$ be a group of cellular automata over alphabet $\Sigma$. We say $G$ has \emph{property~P} if for each $\emptyset \neq F \Subset G$ we can find $g \in F$, and words $|v| = |u|$, $a, b \in \Sigma$, such that for some words $a', b' \in \Sigma$ we have
\[ g((uabv)^\Z)_{\{-1,0\}} \neq g((ua'b'v)^\Z)_{\{-1,0\}} \]
and for all $h \in F \setminus \{g\}$, we have
\[ h((uabv)^\Z)_{\{-1,0\}} = h((ua'b'v)^\Z)_{\{-1,0\}}. \]
%\[ h((uww'v)^\Z)_{\{-R, ..., R-1\}} \neq g((uww'v)^\Z)_{\{-R, ..., R-1\}} \]
%or
%\[ h((uw''w'''v)^\Z)_{\{-R, ..., R-1\}} \neq g((uw''w'''v)^\Z)_{\{-R, ..., R-1\}} \]
\end{definition}

This is easier to understand in terms of actions on conveyor belts (i.e.\ the action on words $(u, v)$ over alphabet $\Sigma^2$ obtained by conjugating them to periodic points $(uv^\bot)^\omega$). In these terms, the definition says that we can always find $g \in F$, and a conveyor belt with $u$ on the top track, and $v$ on the bottom track (in reverse), so that after a suitable modification of the last double letter of the conveyor belt, after applying $g$ we see a change in the initial double letter, but applying any $h \in F \setminus \{g\}$ will not produce a change.

\begin{remark}
\label{rem:Identity}
A few technical points should be made about edge cases: The case $F = \{e_G\}$ holds trivially for any action, by taking $g = e_G$ and $ab, a'b'$ any distinct words of length $2$. On the other hand, if $e_G \in F \neq \{e_G\}$, then the $g$ chosen in the definition can never be the identity element, as we would have to have $u = v = \varepsilon$ (by the inequality involving $g$), and thus the equality involving $h$ implies that $h \neq e_G$ maps two configurations from outside the diagonal into it. If $F$ has at least two elements, for a similar reason we must always choose $u,v$ such that $|u| = |v| \geq 1$.
\end{remark}

Property~P is stronger than unique visits in the following sense.

\begin{lemma}
Let $G$ be a group of automorphisms acting faithfully with property~P on $\Sigma^\Z$. Then the diagonal action $G \curvearrowright (\Sigma^\Z \times \Sigma^\Z)$ has unique visits.
\end{lemma}

\begin{proof}
This is immediate from the definition, using the clopen set $C = \Sigma^2 \setminus \{(s, s) \;|\; s \in \Sigma\}$.
\end{proof}

\begin{lemma}
The group generated by the shift has property~P.
\end{lemma}

\begin{proof}
This is almost immediate from the discussion in the beginning of the section. If $F$ is a finite subset of $\langle \Z \rangle$, let $g = \sigma^k \in F$ be such that $|k|$ is maximal. If $k = 0$ (so $F = \{e\}$) the proof is trivial. Otherwise, suppose $k > 0$ (the case $k < 0$ being symmetric).

Then let $|u| = |v| = k$ be arbitrary, let $w = w' = w''' = 0$, $w'' = 1$. Clearly when applied to a conveyor belt, we have that $g((uw, vw'))_0 \neq g((uw'', vw'''))_0$ because the unique difference $w'' \neq w'''$ travels along the top track of the conveyor belt to the origin. But no other element of $F$ will propagate this difference as far. In particular, we may have $g^{-1} \in F$, but seeing it as a repeated application of $\sigma^{-1}$, the first application of $\sigma^{-1}$ does not propagate it, moves the difference on the bottom track of the conveyor belt, and thus there is no time for it to reach the left end.
\end{proof}

%Note that the choice $h = e_G$ implies that there cannot be any nontrivial symbol permutations in $G$, unless in fact $.

%This definition allows slightly more freedom than the discussion before the definition, namely %. Namely, the two latter formulas are implied by $h((uww'v)^\Z)_{\{-R, ..., R-1\}} = h((uw''w'''v)^\Z)_{\{-R, ..., R-1\}}$, which in turn is a minor generalization of the idea that a
%we do not require the difference to travel at a faster rate in the action of $g$ than in that of any other element of $F$, when applied to some two points (here, $(uww'v)^\Z$ and $(uw''w'''v)^\Z$).

\begin{lemma}
Finite groups admit actions with property~P by automorphisms of a full shift.
\end{lemma}

\begin{proof}
We act on words $w_n = 20^{2n} 1 0^{2m-2n} 3$ freely, for fixed large enough $m$. Let $F$ be any finite subset of the finite group $G$. If $F = \{e\}$ then the claim is trivial with $g = e$, just take any two distinct words of length $2$, say $ab \neq a'b'$, and take $u = v = \varepsilon$.

Otherwise pick any non-trivial element $g \in F$. Suppose $g(w_n) = w_j$ with $n \neq j$, and wrap the word $w_n$ so that it covers part of a conveyor belt (of length at most $2m + 2$ as a double word), and in such a way that the conveyor belt has either $(3,0)$ or $(0,3)$ as its last symbol, $3$ does not otherwise appear in it, and after applying $g$, the leftmost symbol is $(1,0)$ or $(0,1)$ (i.e.\ the $1$-position of the image $w_j$ is at the left end).

Note that after applying any other $h \in G$, the first symbol will be $(0,0)$, since there is a unique appearance of a word $w_i$ (namely $w_n$) in the periodic point corresponding to the initial conveyor belt, the positions of $1$s differ by two cells in different such words, and the action of $G$ is free.

Now if we change the rightmost double symbol $(0,3)$ or $(3,0)$ to $(0, 0)$ in the conveyor belt, then no element of $G$ modifies the conveyor belt, since there are now no occurrences of $3$ in the periodic point it represents, thus no occurrences of double words $w_i$. Thus, this difference in the rightmost double symbol of the conveyor belt travels to the leftmost double symbol in the action of $g$, but does not travel there in the action of any other element, since after any $h \neq g$ is applied, the leftmost double symbol will be $(0,0)$.
\end{proof}

\begin{lemma}
\label{lem:FreeProduct}
Suppose that finitely many groups $(G_i)_{i \in I}$ each act with property~P by automorphisms on a one-dimensional full shift. Then their free product also admits such an action. %acts with property~P %trsly faithfully by automorphisms on any full shift.
\end{lemma}

\begin{proof}
Suppose $I = \{1,2,\dots,\ell\}$. %By Lemma~\ref{lem:FullShiftThenSofic}, it
%We show that the free product acts rasely faithfully on some full shift, and explain .
We consider the construction from the finite case of Theorem~\ref{thm:GraphProducts} (in the special case of free products), and assume the groups $G_i$ all have property~P. %A minor modification is needed to deal with the parameter $R$ we allow, so now $\SSS = (B^R)^{B^R}$ and side-effects should permute length-$R$ suffixes of conveyor belts based on changes in the length-$R$-prefix of the following conveyor belt (note however that $R = 1$ in all our examples, in which case no modification is necessary). % By the previous lemma, we may assume all $g \in G_i \setminus \{e\}$ have minimal radius at least $10$.

If $F \Subset G = G_1 * G_2 * \cdots * G_\ell$, let $m$ be such that the maximal number of syllables in an element of $F$ is $m$, when written in irreducible form.

We start with an overview. In the proof, we will apply essentially the proof of Theorem~\ref{thm:GraphProducts} (specialized to free products) to a specific element $g = g_1 \cdots g_m \in F$ (coming from an irreducible word $w = w_1 \cdots w_m \in (\bigcup_i G_i)^*$), to find a pair of words $t, t'$ which consist of segments corresponding to syllables of $w$, which only differ in the rightmost double symbol, and when $g$ is applied diagonally to $(t, t')$, the leftmost difference moves faster than in the application of any other element of $F$.

Since we are dealing with a free product, analogously to as how the proof Theorem~\ref{thm:GraphProducts} specializes to this situation, we will take exactly one segment per syllable. If we index the segments from left to right as $1, \ldots, m$, then we will specifically ensure that after applying $g^j = g_j \circ \cdots g_m$ to the pair $(t, t')$, the leftmost segment where in $g^j(t)$ and $g^j(t')$ there is a difference in the initial double symbols is going to be segment $j$.

Let us now get into more detail. We need one minor modification in the construction of Theorem~\ref{thm:GraphProducts}, namely to add a ``wall symbol'' $\#$ which is never modified, and whose role is to allow us to cut segments artificially (this is needed at the very end). Such a symbol can of course be added directly, and the proof goes through verbatim. Alternatively we can add a group $G_0$ commuting with all other groups, and then finally ignore its action -- the symbols corresponding to segments of that group will now act as walls.

The construction of the pair $t, t'$ takes its inspiration from the proof Lemma~\ref{lem:Product}. We can start by fixing any type $u = \tau(w) \in \{1, \ldots, \ell\}^m$ of the maximal length $m$, such that $w$ is irreducible and represents an element $g \in F$. Consider now the set of all words $w$ with $\tau(w) = u$, representing an element of $F$, and let $F_m$ be the set of all group elements obtained as $w_m$ as $w$ ranges over this set. Note that $e_G \notin F_m$ because we consider only irreducible words.

Since $F_m \cup \{e_G\}$ is a finite nonempty subset of $G_{u_m}$, we can apply property~P to obtain $g_m \in F_m$ (by Remark~\ref{rem:Identity} indeed $g_m \neq e_G$) such that some pair of conveyor belts $t^m, t'^m$ of the same length only differ in the last double symbol, and after applying $g_m$ they differ in their first double symbol, but they never differ in their first double symbol after applying another element of $F_m \cup \{e_G\}$. Furthermore, $t^m$ and $t^{m'}$ are of length at least $2$ as double words.

These will be the suffixes of $t, t'$ respectively. Now restrict $F$ to elements of length $m$ which have syllable structure $u$, and inductively continue on this set, analogously to the proof of Lemma~\ref{lem:Product}, to build pairs $t^j, t'^j$ and group elements $g_j$ by applying property~P on each step. (These will not directly be subsegments of $t, t'$, however.)

We will pick $g = g_1 g_2 \cdots g_m$ as our final element. To construct the words $t, t'$, we note as in the proof of Theorem~\ref{thm:GraphProducts} that because information cannot flow to the right, it is possible to ensure, by picking suitable initial contents for each segment and by picking suitable $\SSS$-symbols between segments, that when applying elements $h^j = g_j \circ \cdots \circ g_m$ to the pair $(t, t')$ with decreasing $j$, the first time a difference appears in the left double symbol of the $j$th segment is indeed precisely immediately after $h^j$ is applied, and at this point (if $j > 1$), the contents of the $(j-1)$th segment is precisely $t^{j-1}, t'^{j-1}$ due to the side-effect making a change in the rightmost symbol of that segment.

In particular, then in the action of $g$, a difference in the last double symbols of $t, t'$ propagates to the leftmost double symbol. We argue that the same does not happen for any $h \in F \setminus \{g\}$. Namely, suppose $h \in F \setminus \{g\}$ is represented by $w'$. Write $h = h_1 \circ \cdots \circ h_{m'}$ where $h_i$ correspond to syllables (i.e.\ come from different groups $G_j$ and are nontrivial).

%Then when applying the $h_j$ in decreasing order of $j$, we see that the smallest index $i$ of a segment such that either there is a difference in the leftmost double symbol of the $i$th segments of $t, t'$, or in the rightmost double symbols of the $(i-1)$th segments, can decrease by at most $1$ on each step.

%the leftmost segment where $t, t'$ differ in the leftmost symbol can move to the left by at most one (in index), because a pair of equal segments in $t, t'$ will stay equal, and thus the only possibility is that . A side-effect may propagate to the segment on the left, it can only propagate to the rightmost symbol, and each segment consists of at least two double symbols). In fact, more than this, we see that if there are $\ell$ applications $h_j$ where.

%In particular .

Suppose first $\tau(w') \neq u$. We have $|\tau(w')| \leq m$, so we can write $\tau(w') = u' = v'bq$, $u = vaq$ for some $q$, with $a \neq b$ and $|v'| \leq |v|$. Now it is easy to see that after applying syllable by syllable, after applying the syllables corresponding to $q$, the leftmost segment where $t, t'$ could possibly differ in the leftmost double symbol after the application of $h$ is $m - |q| + 1$ (and they may or may not also differ in the rightmost symbol of segment $m - |q|$), with the interpretation that if $|q| = 0$ then the words are precisely $t,t'$ and only differ in the rightmost double symbol. The segment $m - |q|$ has type $a$, so $b$ will not modify it, and there is no time for the difference to reach the leftmost double symbols in the remaining application of the elements $h_i$ comprising $h$.

Suppose then that $h = h_1 \circ \cdots \circ h_m$ has the same syllable structure $u$ as $g$. Since it is not actually the same element, at least one syllable has to be different, and there is a rightmost such syllable. Now, we can apply a similar argument as in the previous paragraph: suppose $h_j \neq g_j$ but $h_i = g_i$ for $i > j$, then $k = h_{j+1} \circ \cdots \circ h_m = g_{j+1} \circ \cdots \circ g_m$.

By construction, in $k((t, t'))$ the $j$th segments contain the two conveyor belts $t^j, (t')^j$ proving property~P for the group $G_{u_j}$ (see the description of property~P in terms of conveyor belts below Definition~\ref{def:P}), specifically applied to a finite set of group elements including $h_j$, and $g_j$ is the unique element in this set that propagates a difference to the left end when applied to $(t^j, (t')^j)$. Thus, at this point $h_j$ fails to propagate the difference at the maximal rate, and the remaining applications of $h_i$, $i < j$, cannot catch up.

To show that we have property~P, observe that the discussion above is in terms of an action on finite words, and that when applied to configurations, the words will behave the same way unless the context lengthens the left or rightmost segments, or produces side-effects. Thus, if we take the two words $t, t'$, consider them as the top tracks of a two conveyor belt, and on the bottom track of the conveyor belts put $\#^{|t|}$, so that indeed the top tracks behave as finite words, then they prove property~P for the set $F$.
\end{proof}

We see no reason why this could not be generalized to infinite free products analogously to Theorem~\ref{thm:GraphProducts}, but have opted for the simpler statement. %.we believe the action constructed there is in fact resly faithful.

The previous results immediately give the following:

\begin{corollary}
The free product of any finite family of finite groups and copies of $\Z$ acts with property~P (thus with unique visits, thus without $A$-cancellation for all $A$) by automorphisms of a full shift.
\end{corollary}

The fact finite free products act faithfully on full shifts was first proved in \cite{Al88}, and possibly this construction already has property~P.

%\begin{proof}

%It is enough to construct an action of an arbitrary finite group $G$ with the property from the previous theorem. Let $|G| \in \{n, n-1\}$ for $n$ even, take alphabet $\{0,1,2,3,4\}$ and let $G$ act on the words $w_i = 1 0^{2i} 2 0^{2n-2i-1} 3$ by the left regular action, identifying $w_i$ with the $i$th group element, under an arbitrary identification (starting with $w_0$, and ignoring $w_{|G|}$ if $n = |G|+1$).
%It is easy to see that property~P holds: by changing $3$ to $4$ we can ``cancel'' the application of $g \in G \setminus \{e_G\}$, so all that is needed is to pick $w_i$ so that the action of $g \in G$, if not cancelled, changes $w_i$ to $w_{n/2}$. Since the action is free, no other $h \in G$ changes $i$ to the same value (or a neighboring value, because the $1$s are separated by distance two), whether or not it is cancelled.
%\end{proof}

Of course, $\Z$ embeds in a free product of finite groups, so its inclusion is superfluous in the previous corollary. The following is also immediate.

\begin{corollary}
Every free group acts with property~P (thus with unique visits, thus without $A$-cancellation for all $A$) by automorphisms on a full shift.
\end{corollary}

%\begin{proof}
%The free product $\Z_2 * \Z_2 * \Z_2$ contains all free groups, and by Lemma~\ref{lem:Subgroup} a subaction of a rsely faithful action is strongly faithful.
%\end{proof}

\section{Wreath products}

\begin{theorem}
\label{thm:WreathProducts}
Let $A$ be a finite abelian group and suppose $G \leq \Aut(\Sigma^\Z)$ acts with non-$A$-cancellation. Then $A \wr G \in \mathcal{G}$.
\end{theorem}

\begin{proof}
Suppose $G$ and $A$ are disjoint. Let $\theta : \Sigma^\Z \to \Hom(A,B)$ be the proof of non-$A$-cancellation. We pick the alphabet $\Sigma \times B$, and to $g \in G$ associate $\hat g \in \Aut((\Sigma \times B)^\Z)$ by $\hat g(x,y) = (gx, y)$; clearly $g \mapsto \hat g$ is a homomorphism. To $a \in A$ associate $\hat a$ by
\[ \hat a(x, y)_i = (x_i, \theta(\sigma^i(x))(a) + y_i) \]
The map $a \mapsto \hat a : A \to \Aut((\Sigma \times B)^\Z)$ is a homomorphism because
\begin{align*}
\widehat{(a + b)}(x, y)_i &= (x_i, \theta(\sigma^i(x))(a + b) + y_i) \\
&= (x_i, \theta(\sigma^i(x))(a) + \theta(\sigma^i(x))(b) + y_i) \\
&= \hat b(\hat a(x, y))_i.
\end{align*}
We have
\[ \hat a^{\hat g}(x, y)_i = (x_i, \theta(\sigma^i(gx))(a) + y_i), \]
from which it is clear that any $\hat a^{\hat g}$ and $\hat b^{\hat g'}$ commute for $a,b \in A, g,g' \in G$. These are the standard relations of the wreath product, and thus we have obtained an action of the wreath product.

We need to show the faithfulness of this action. For this, suppose $w \in A \wr G$ and $\hat w = \id_{(\Sigma \times B)^\Z}$. We will show that $w$ is indeed trivial. Clearly the natural homomorphism from $A \wr G$ to $G$ maps $w$ to the identity, as $G$ acts by its original non-$A$-cancelling (thus faithful) action on the $\Sigma^\Z$-component of our full shift $(\Sigma \times B)^\Z)$.

Thus, we can write $w = \prod_{j = 1}^n a_j^{g_j}$ for some $n$, and some $g_j \in G, a_j \in A$, where we may assume $g_j \neq g_j$ for $i \neq j$ and $a_i \neq 0$ for all $i$. By the formula for conjugates $\hat a^{\hat g}$, we have
\[ \hat w(x, y)_i = (x_i, \sum_{j = 1}^n \theta(\sigma^i(g_j x))(a_j) + y_i) = (x_i, y_i) \]
for all $x, y$, so $\sum_{j = 1}^n \theta(\sigma^i(g_j x))(a_j) = 0$ for all $x \in \Sigma^\Z$ and $i \in \Z$. Setting $i = 0$ we have
\[ \sum_{j = 1}^n \theta(g_j x)(a_j) = \sum_{g \in G} \theta(g x)(f(g)) = 0 \]
where $f = \sum_{j = 1}^n a_j \cdot g_j$, so by non-$A$-cancellation $f = 0$, meaning $w = \id$. Thus $w \mapsto \hat w$ is injective.
\end{proof}

Combining with the results of the previous section, we obtain that for any finite abelian group $A$ and any $n \in \N$, we have $A \wr F_n \in \mathcal{G}$ and $A \wr \Z^n \in \mathcal{G}$. %It is plausible that $A$ can be generalized to other (necessarily abelian) groups.

While we state the theorem only in the setting of automorphism groups of full shifts, the proof applies more generally:

\begin{example}
Let $X$ be the Ledrappier subshift, and let $Y \subset \{0,1,2\}^{\Z^2}$ be the set of configurations where $y_i\neq 0 \wedge y_j \neq 0 \implies i = j$. We showed in Example~\ref{ex:Ledra} that the shift action on $X$ is non-$A$-cancelling for $A = \Z_2$. Analogously to the above proof we obtain $\Z_2 \wr \Z^2 \leq \Aut(X \times Y)$. \qee
\end{example}

The method of this section cannot cover groups like $\Z \wr \Z$, $\Z_2 \wr (\Z_2 \wr \Z)$, suggesting the following questions. %These are solved by a different method in the follow-up paper \cite{Sa}.

\begin{question}
Is $\Z \wr \Z$ in $\mathcal{G}$?
\end{question}

%The proof that (for instance) $\Z_2 \wr \Z$, requires that the action of $\Z$ is $Ars, and the action of $\Z_2 \wr \Z$ we obtain is typically not $Ares. This suggests the following candidate group to study. Note that it is residually finite by \cite[Theorem~3.1]{Gr57} and has word problem solvable in polynomial time.

\begin{question}
Is $\Z_2 \wr (\Z_2 \wr \Z)$ in $\mathcal{G}$?
\end{question}

We solve these in a follow-up paper \cite{Sa23} (using a different method) and show that even $\Z_2 \wr (\Z_2 \wr \Z)$ and $\Z \wr (\Z_2 \wr \Z)$ are in $\mathcal{G}$.

\section{The one-sided case}

\begin{theorem}
\label{thm:OneSidedGraphProducts}
For a finite graph $\Gamma$ and groups $G_i \in \mathcal{G}_{n_i}'$ we have $(G_1,G_2,\ldots,G_k)^\Gamma \in \mathcal{G}_{\max n_i + 1}'$.
\end{theorem}

\begin{proof}
The construction for finite $\Gamma$ in the proof of Theorem~\ref{thm:GraphProducts} works almost directly, the main difference being that we replace the conveyor belts with words. We see $\Aut(\Sigma^\N)$ as the subgroup of $\Aut(\Sigma^\Z)$ containing those $f$ such that both $f$ and $f^{-1}$ are one-sided as automorphisms of $\Sigma^\Z$.

First, we give a construction with a large alphabet. Suppose $|B_i| = n_i$ % = \{0,1,\dots,n_i-1\}$ (where we see $B_i$ and $B_j$ as disjoint even if $n_i = n_j$),
with $B_i \cap B_j = \emptyset$ for $i \neq j$ and let $\Sigma = \SSS \sqcup \bigsqcup_i B_i$ where $\SSS$ is the disjoint union of sets of functions $B_i^{B_j}$. If $g \in G_i$, on segments of type~$i$ (maximal words over the subalphabet $B_i$) $\hat g$ applies the natural action, interpreting the last symbol as a constant tail: the segment contents $ua \in B_i^{\ell}, a \in B_i$, is replaced by $vb \in B_i^{\ell}, b \in B_i$, where $g(ua^\infty) = vb^\infty$.

On segments of type~$j$, $j \neq i$, we modify the last symbol as in Theorem~\ref{thm:GraphProducts} when the leftmost symbol of a segment of type~$i$ is changed: if $c$ is changed to $d$ when $\hat g$ is applied, then in the segment of type~$j$ on the left we change $a$ to $a + s(d) - s(c)$. It is clear that $\hat g$ is a one-sided automorphism, and the proof of correctness is analogous to that in Theorem~\ref{thm:GraphProducts}.

Now, let us optimize the alphabet size. First, we need not actually have the individual symbols know the type of the segment they belong to, as long as they can tell their type by looking to the right. Thus, we can take an alphabet of size $\max_i n_i + 1$, use the same symbols for all segments, and use the one extra symbol, say $\#$, to denote the type of the segment, representing $w \in B_i^*$ by $w_0 \#^i w_1 \#^i \cdots \#^i w_{|w|-1} \#^i$.

Similarly, the large set $\SSS$ can be replaced by words of the form $\#^{m k}$ for $m = 1, 2, \dots, |\SSS|$, since $i \in \{0,1,\dots,k-1\}$ and $m$ can be deduced from $\#^i \#^{m k}$ (interpreting $m = 0$ as the lack of an $\SSS$-symbol).
\end{proof}

Apart from some trivial cases, we do not know if increasing the alphabet is necessary, nor whether it is possible to do infinite products.

\begin{remark}
\label{rem:Reduction}
The radii of the automorphisms are rather massive in this construction. If desired, this can be avoided by using two special symbols instead of one, and coding the lengths of the $\#^*$-runs in binary, giving embeddings of graph products of groups $G_i \in \mathcal{G}_{n_i}'$ in $\mathcal{G}_{\max n_i + 2}'$ with more reasonable radii. With a bit more work, one can optimize both the alphabet size and the radii simultaneously, by using the basic size $\max n_i$ alphabet also for coding the lengths of runs. \qee
\end{remark}

\begin{theorem}
Every right-angled Coxeter and Artin group is in $\mathcal{G}'_3$.
\end{theorem}

\begin{proof}
We have $\Z_2 \in \mathcal{G}_2'$, since $f(x)_i = 1-x_i$ defines an automorphism of the binary full shift. If $\Gamma$ is a graph, then $(\Z_2,\Z_2,\dots,\Z_2)^\Gamma$ is, by definition, just the right-angled Coxeter group $C(\Gamma)$. The previous theorem then shows that that $\Aut(\{0,1,2\}^\N)$ contains every right-angled Coxeter group. Right-angled Artin groups are clearly subgroups of right-angled Coxeter groups (see \cite{DaJa00} for a finite-index embedding).
\end{proof}

\begin{question}
Which Coxeter groups (or Artin groups) are in $\mathcal{G}'_n$ for each $n$?
\end{question}

There is a bound on orders of finite-order elements \cite{BoFrKi90}, which forbids some Coxeter groups, but it is likely that there are additional restrictions.

\begin{example}
As right-angled Coxeter and Artin groups are important families of groups, we explain the construction explicitly for RACGs $(G_0,G_1,\ldots,G_{k-1})^\Gamma$, where $G_i \cong \Z_2$ for all $i = 0, \dots, k-1$. Then $\hat g (x)_t = a$ is determined as follows: If $x_t = 2$, then $a = 2$. If $x_t \in \{0,1\}$, and the maximal segment of $2$s to the right is $2^j$, then let $j = i' + mk$ where $0 \leq i' < k$ and $m \in \{0,1\}$. If this cannot be done since the $2^*$-segment is too long, then set $a = x_t$.

Suppose then that we find $i'$ and $m$. Then if $i' = i$, set $a = 1-x_t$ (we are inside a segment of type~$i$). If $i' \neq i$ and $m = 0$, set $a = x_t$ (we are properly inside a segment of type~$i'$). If $i' \neq i$, $m = 1$ and $\{i, i'\} \in E$ set $a = x_t$ (we are at the right boundary of a segment of type~$i'$, but $G_i$ commutes with $G_{i'}$).

If $i' \neq i$, $m = 1$ and $\{i, i'\} \notin E$, then consider the configuration to the right of the $2^j$-word. It starts with an element of $\{0,1\}$, so suppose it is in $\{0,1\}2^{j'}$ for some maximal $j'$ (again if the run is very long, set $a = x_t$). Let $j' = i'' + m' k$. If $i'' \neq i$, then set $a = x_t$. If $i'' = i$ then check whether $\hat g$ will flip this bit. Set $a = 1 - x_t$ if it will, otherwise set $a = x_t$.

Analogously to the proof of Lemma~\ref{lem:FreeProduct}, it is easy to see that the actions has unique visits when $\Gamma$ has no edges. \qee
\end{example}

\begin{remark}
\label{rem:YesCanBeSF}
In particular, letting $\Gamma$ be the complement of the disjoint union of $n$ copies of the $k$-clique in the above theorem, we get
\[  \Z^n, F_k \leq F_k^n \leq (\Z_2 * \Z_2 * \cdots * \Z_2)^n = (\Z_2,\Z_2,\dots,\Z_2)^\Gamma \in \Aut(\{0,1,2\}^\N). \]
This action does not have unique visits in general, but similarly as in Section~\ref{sec:SFActions}, we see that in this special case, when applied to an edgeless graph, we get an action with unique visits for the free product of copies of $\Z_2$. As in Lemma~\ref{lem:Product} it is easy to see that direct products of actions with unique visits have unique visits. From these observations, we get that $\Z^n$ and $F_k$ admit actions with unique visits in $\mathcal{G}'_\infty$. \qee
\end{remark}

\begin{theorem}
\label{thm:OneSidedWreathProducts}
If $A$ is finite abelian and $G \in \mathcal{G}'_\infty$ has non-$A$-cancellation, then $A \wr G \in \mathcal{G}'_\infty$.
\end{theorem}

\begin{proof}
Let $N = \{0,1,\dots,n-1\}$ and suppose $G$ acts with non-$A$-cancellation on $N^\N$ by automorphisms. Let $\theta : N^\N \to \Hom(A, B)$ be the continuous function from the definition of non-$A$-cancellation. For our embedding we pick the alphabet $\Sigma = B \sqcup N$. If $g \in G$, $\hat g$ acts on maximal finite segments over $N$ as in the above proof (so the word $wa \in N^*$ with $a \in N$ represents $wa^\infty$, and we conjugate the natural action through this identification). Symbols in $B$ are not modified by $\hat g$.

If $a \in A$, $\hat a$ acts trivially on symbols in $N$. If $x_i \in B$, then if also $x_{i+1} \in B$ we set $\hat a(x)_i = x_i$, while if $x_{i+1} \in N$ we interpret a prefix of $x_{[i+1, \infty)}$ as a configuration in $y \in N^\N$ the same way $\hat g$ does, i.e.\ if $x_{[i+1, \infty)}$ is a one-way infinite word over $N$, then directly use this configuration, and otherwise take the maximal finite segment over $N$ and interpret the last symbol as being repeated infinitely. Now set $\hat a(x)_i = x_i + \theta(y)(a)$. The proof that this gives an embedding of the wreath product is similar to Theorem~\ref{thm:WreathProducts}.
\end{proof}

The non-$A$-cancellation certificate of $G$ might use a large group $B$, so it is difficult to include alphabet sizes in the statement. The proof, combined with ideas from the proof of Theorem~\ref{thm:OneSidedGraphProducts}, shows that $\max_i p_i^{e_i} + 1$ symbols suffice in addition to $N$ if $A = \prod_i \Z_{p_i}^{e_i}$, because we can use $B = A^d$ (Lemma~\ref{lem:BAd}) which acts faithfully by disjoint cycles on a finite set of cycles of cardinalities in $p_i^{e_i}$, and because we can code the type of the cycle using the alphabet itself.

\begin{theorem}
If $A$ is finite abelian and $n \in \N$, then $A \wr F_n, A \wr \Z^n \in \mathcal{G}'_\infty$.
\end{theorem}

\begin{proof}
This follows from the previous theorem and Remark~\ref{rem:YesCanBeSF}.
\end{proof}

\bibliographystyle{plain}
\bibliography{bib}{}

\end{document}